\documentclass[10pt,reqno]{amsart}
\usepackage{amsfonts,amsmath,amstext,amsbsy,amsopn,amsthm,amscd}

\usepackage{graphicx}
\usepackage{color}
\usepackage{tikz-cd} 
\newtheorem{theorem}{Theorem}[section]
\newtheorem{lemma}[theorem]{Lemma}%[section]
\newtheorem{corollary}[theorem]{Corollary}%[section]
\newtheorem{proposition}[theorem]{Proposition}%[section]
\theoremstyle{definition}

\newtheorem{definition}[theorem]{Definition}%[section]
\newtheorem{assu}[theorem]{Assumption}%[section]
\newtheorem{exmp}[theorem]{Example}%[section]
\newtheorem{rem}[theorem]{Remark}%[section]

\theoremstyle{remark}

\numberwithin{equation}{section}

\newcommand{\field}[1]{\mathbb{#1}}

\newcommand{\R}{\field{R}}
\newcommand{\C}{\field{C}}

\newcommand{\bbW}{\field{W}}

\DeclareMathOperator{\dom}{dom}    %domain of an operator
\DeclareMathOperator{\dist}{dist}  %distance
\DeclareMathOperator{\supp}{supp}  %support of a function
\DeclareMathOperator{\sign}{sign}  %signum 
\DeclareMathOperator{\dive}{div}   %divergence

\newcommand{\abs}[1]{\lvert#1\rvert}

\newcommand{\norm}[1]{\lVert#1\rVert}

\newcommand{\calL}{\mathcal{L}}
\newcommand{\calA}{\mathcal{A}}
\newcommand{\calM}{\mathcal{M}}
\newcommand{\calP}{\mathcal{P}}
\newcommand{\calQ}{\mathcal{Q}}

\newcommand{\bbA}{\field{A}}
\newcommand{\dx}{\,\mathrm{d} x}
\newcommand{\wt}{\widetilde}

\newcommand{\bfC}{\mathbf{C}}
\newcommand{\bbH}{\field{H}}

\DeclareMathOperator{\dskew}{skew}
\DeclareMathOperator{\axl}{axl}

\begin{document}

\author{Robert Haller-Dintelmann}
\address{Technische Universit\"at Darmstadt, Fachbereich
Mathematik, Schlossgartenstr.\@ 7, D-64298 Darmstadt, Germany}
\email{haller@mathematik.tu-darmstadt.de}
	
\author{Alf Jonsson}
\address{Umea universitet SE-901 87 Umea Sverige}
\email{alf.jonsson@math.umu.se}

\author{Dorothee Knees}
\address{Weierstrass Institute for Applied Analysis and Stochastics,
 Mohrenstr.\@ 39, D-10117 Berlin, Germany}
\email{dorothee.knees@wias-berlin.de}

\author{Joachim Rehberg}
\address{Weierstrass Institute for Applied Analysis and Stochastics,
 Mohrenstr.\@ 39, D-10117 Berlin, Germany}
\email{rehberg@wias-berlin.de}

\title[Elliptic and parabolic regularity]{Elliptic and parabolic regularity for
second order divergence operators with mixed boundary conditions}			
% in Englisch: Kleinschreibung, nach Doppelpunkt oder Bindestrich mit Grossbuchstaben beginnen
%\nopreprint{xxxx}	% Preprint-Nummer
%%\nopreyear{2010}	% Jahr des Preprints

%\selectlanguage{english}		% hier nicht veraendern, wichtig fuer Datumsformat
%%\date{21 Dec 2010}			% Datum fixieren - Schreibweise z.B. October 7, 2009
\subjclass[2010]{
35B65, %Smoothness and regularity of solutions
35J47, %Second-order elliptic systems
35J57, %Boundary value problem
%35Q74 %PDEs in connection with mechanics of deformable solids
%49N60 %Regularity of solutions
74B05%Classical linear elasticity
}	% Math. Subject Classif.
%\pacs[2008]{}				% ggf. Physics Astronomy Classif.
\keywords{mixed boundary conditions, interpolation, elliptic
  regularity for equations and systems, analytic semigroups } 
%\thanks{This work was partially supported by the German Research
% Foundation {\sc Matheon}}
				% Danksagungen; Punkt am Satzende erscheint automatisch!
%% amsart only! abstract befor maketitle
%
\begin{abstract}
We study second order equations and systems on non-Lipschitz domains including mixed boundary
conditions. The key result is interpolation for suitable function spaces. From this, 
elliptic and parabolic regularity results are deduced by means of Sneiberg's isomorphism theorem.
\end{abstract}
\maketitle
%% article and other classes: abstract after maketitle
%\begin{abstract} \end{abstract}
%%%%%%%%%%%%%%%%%%%%%%%%%%%%%%%%%%%%%%%%%%%%%

\section{Introduction}
In this paper we first establish interpolation properties for function spaces
that are related to mixed boundary value problems. Afterwards, from this and a
fundamental result of Sneiberg \cite{snei} (cf. also \cite{vigna}) we deduce 
elliptic and parabolic regularity results for second order, divergence operators.

In recent years it became manifest that the appearance of mixed boundary
conditions is not the exception when modelling real world problems, but more
the rule. For instance, in semiconductor theory, models with only pure
Dirichlet or pure Neumann conditions are meaningless, see \cite{selberherr84}.

One geometric concept, which proved of value for the analysis of mixed boundary
value problems, is that introduced by Gr\"oger in \cite{groeger89} (compare
also \cite{mitr} and references therein). It demands, roughly speaking, that
the domain $\Omega$ under consideration is a Lipschitz domain and that the 'Dirichlet
part' $D \subset \partial \Omega$ of the boundary is locally separated from the
rest by a Lipschitzian hypersurface within $\partial \Omega$. Within this
geometric framework, several properties for differential operators,
well-known from smooth constellations, were re-established. This concerns
elliptic regularity (in particular H\"older continuity) \cite{groeger89},
\cite{griehoel}, \cite{grie/reck}, \cite{h/m/r/s}, maximal parabolic regularity
\cite{griemax}, \cite{hal/reh}, \cite{hal/reh1} and interpolation \cite{ggkr}.

In this paper, we impose more general conditions on the domain and on the
Dirichlet boundary part $D$; notably, we dispense the Lipschitz property of the
domain. In particular, the domain may touch itself from outside, see the
examples in Figures~\ref{fig-Meissel-1} and \ref{fig-Meissel} below that are
included in our framework. Note that the situation in Figure~\ref{fig-Meissel}
 is not an artificial one:
the reader may think of a body for which $\Sigma$ and the two striped areas
form an extremely thin, but highly conducting contact $D$, to which an external
 source (e.g.\@ heat or electrical) is applied. If the body is formed by a much less conducting
material, the distribution of heat/charge within the body is subject of an
elliptic/parabolic equation with Dirichlet conditions on $D$.

\begin{figure}
\begin{minipage}[t]{.47\textwidth}
 \centerline{\includegraphics[scale=0.7]{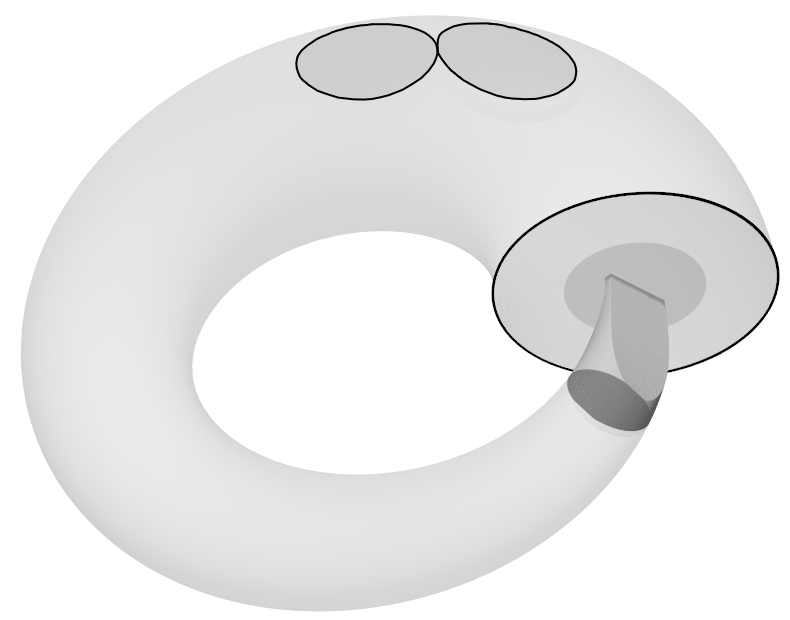}}
 \caption{ A geometric non-Lipschitzian setting which fulfills our assumptions,
	if the grey apex and the three shaded circles carry the Dirichlet
	condition.}
 \label{fig-Meissel-1}
\end{minipage}
\begin{minipage}[t]{.47\textwidth}
  \centerline{\includegraphics[scale=.55]{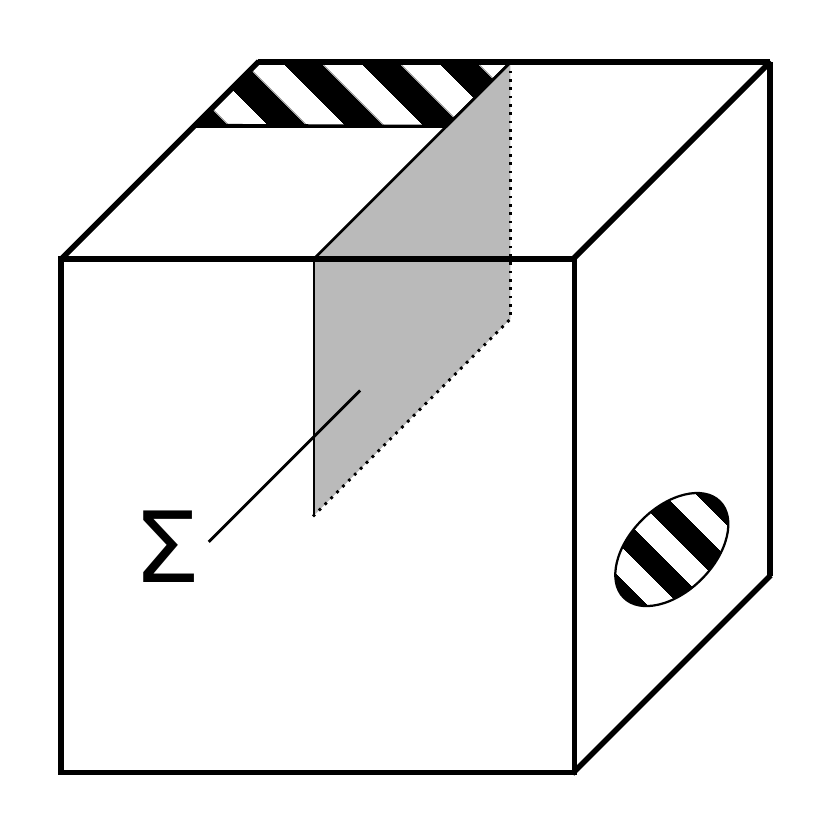}}
  \caption{The set $\Sigma$ does not belong to $\Omega$, and carries -- together
	with the striped parts -- the Dirichlet condition.}
  \label{fig-Meissel}
\end{minipage}
\end{figure}

Our geometric framework is the following. The Dirichlet boundary part $D$ only
has to be a $(d-1)$-set in the sense of Jonsson/Wallin. This can be seen as an
-- extremely weak -- compatibility condition between $D$ 
and $\partial \Omega \setminus D$. For the complement of the Dirichlet boundary
part, the crucial feature is the local extendability of Sobolev functions.
%{\bf Resultate paraphrasieren (Interpolation und Regularit\"at)}
Within this geometrical framework we prove the following: the spaces
$W^{1,p}_D(\Omega)$ (cf. Definition \ref{d-1}) $ p \in ]1,\infty[$ interpolate
according to the same rules as if one \emph{formally} replaces the domain
$\Omega$ by a ball $B$ and the boundary part $D \subset \partial \Omega$
by the empty set (compare \cite[Ch.~4.3.1]{triebel}).
Based on the interpolation results we reproduce Gr\"oger's elliptic regularity
result from \cite{groeger89}, namely that an arbitrary elliptic divergence operator
$-\nabla \cdot \mu \nabla + 1$ provides a topological isomorphism between a
space $W^{1,p}_D(\Omega)$ and $W^{-1,p}_D(\Omega)$ for $p$ close to $ 2$ -- but now
for a much broader class of domains and Dirichlet boundary parts.
Let us emphasize that the -- matrix valued -- coefficient function $\mu$ of the
operator needs only be bounded and elliptic, cf.\@ Assumption~\ref{a-coeff} below.
Note that the main result from \cite{groeger89} was used in some tens of papers in
order to treat (mostly two dimensional) problems, stemming from real world
applications. Having this regularity result at hand, we succeed in proving
that divergence operators of this type generate analytic semigroups on spaces
$W^{-1,p}_D(\Omega)$, as long as $p$ is chosen close to $2$. Clearly, this can
serve as the adequate instrument for the treatment of corresponding parabolic
problems, compare e.g.\@ \cite{amannteubner}, \cite[Ch.~2]{henry}, \cite{luna}.

One of our main technical tools is the version of the now classical
restriction/extension theorem of Jonsson/Wallin (\cite[Ch.~V.1 Thm.1]{jons})
for the limit case of Lipschitz functions, see Proposition \ref{p-JW} below.

Throughout we stick to the condition that $D$ is a $d-1$-set, which in several
instances can in fact be weakened. Since our motivation for this paper comes
from the applications, our aim is to describe a very general but nevertheless
easily accessible geometric constellation that allows to deduce our results.

The outline of the paper is as follows: in the next section we introduce some
preliminaries. In Section~\ref{sec3} we reproduce interpolation within the
family of spaces $\{W^{1,p}_D(\Omega)\}_{p \in {]1,\infty[}}$, and, as a
consequence, in $\{W^{-1,p}_D(\Omega)\}_{p \in ]1,\infty[}$. Rather
unexpectedly, this follows directly from the results of Jonsson/Wallin,
combined with a classical interpolation principle for complemented subspaces
and the existence of an extension operator $\mathfrak E : W^{1,p}_D(\Omega) \to
W^{1,p}_D(\R^d)$, which is uniform in $p$.

Since the existence of an extension operator is thus crucial for our approach,
in Section~\ref{extens} we first establish construction principles for
extension operators. These, together with our conditions on the geometry of
$\Omega$ near $\partial \Omega \setminus D$, then indeed assure their
existence. A simple 'pre-processing', which essentially improves the
applicability of our setting, is described in Lemmas~\ref{l-modigeb} and
\ref{l-exten0}. It allows to pass from the original domain $\Omega$ to another
domain $\Omega_\bullet \subseteq \overline \Omega$ whose boundary is smaller
and, in most cases, a more regular one. It is exactly this what enables also
the treatment of geometric settings like in Figure~\ref{fig-Meissel}, compare
\cite{mitrea} for a similar, but different approach -- there even applied to
higher order Sobolev spaces.
Section \ref{sec4} contains the above mentioned elliptic and parabolic 
regularity results. 
In Section~\ref{sec_systems} we extend the discussion to a class of elliptic
systems comprising the equations for linear elasticity and for Cosserat models.
Relying on the interpolation results it is shown that the corresponding
differential operators provide topological isomorphisms between
$\bbW_D^{1,p}(\Omega)$ and $\bbW_D^{-1,p}(\Omega)$ for suitable $p>2$.
Moreover, under an additional symmetry assumption on the coefficient tensor,
uniform estimates are derived for classes of coefficient tensors satisfying
certain uniform bounds. Since in the case of systems the coercivity of the
operator not necessarily entails the positivity of the coefficient tensor, the
pointwise arguments from \cite{groeger89} have to be modified and transferred
to arguments dealing with the whole operator. In this way also the results from
\cite{HMW11} are extended to more general geometric situations.

Finally, in Section~\ref{sec-appl}, we point out a broad class of possible
applications for our regularity results.

\section{Notation, Preliminaries}

If $X$ and $Y$ are two Banach spaces, then we use the symbol $\mathcal L(X;Y)$
for the  space of linear, continuous operators from $X$ to $Y$. In case of
$X=Y$ we abbreviate $\mathcal L(X)$.

We are now going to impose the adequate condition on the Dirichlet boundary
part $D$. For this we first recall the notion of an $l$-set,
cf.\@ Jonsson/Wallin \cite[II.1.1/2]{jons}.

\begin{definition} \label{d-lset} Assume $0 < l \le d$.
Let $M \subset \R^d$ be closed and $\rho$ the restriction of the $l$-dimensional
Hausdorff measure $\mathcal H_{l}$ to $M$. Then $M$ is called an $l$-set, if there 
exist two positive constants $c_1, c_2$ that satisfy
\begin{equation} \label{e-0}
 c_1 r^l \le \rho \bigl( B(\mathrm x,r) \cap M \bigr) \le c_2 r^l, \quad
	\mathrm x \in M, r \in {]0,1[},
\end{equation}
where $B(\mathrm x,r)$ is the ball with center $\mathrm x$ and radius $r$ in
$\R^d$.
\end{definition}

\begin{assu} \label{a-1}
 Let $\Omega \subset \R^d$ always be a bounded domain and let $\Gamma$ be an
 open part of $\partial \Omega$, such that $D := \partial \Omega \setminus
 \Gamma$ is a $(d-1)$-set.
\end{assu}

We now define the adequate Sobolev space of first order that reflects the
Dirichlet condition.

\begin{definition} \label{d-1}
Let $\Lambda \subset \R^d$ be a domain and let $F$ be a closed subset of
$\overline{\Lambda}$. Then we define
\begin{equation} \label{e-closuresp}
 C^\infty_F(\Lambda) := \{\psi|_\Lambda : \psi \in C^\infty(\R^d),
	\ \supp(\psi) \cap F =\emptyset \}.
\end{equation}
Moreover, for $p \in [1,\infty[$, we denote the closure of
$C^\infty_F(\Lambda)$ in $W^{1,p}(\Lambda)$ by $W^{1,p}_F(\Lambda)$.
\end{definition}
In particular, the set $F$ may be identical with the boundary part $D$.\\
Since the ultimate instrument for almost everything in the next section is a
classical result of Jonsson/Wallin (see \cite[Ch.~VII]{jons}) we quote this
here for the convenience of the reader.
\begin{proposition} \label{p-JW}
 Let $F \subset \R^d$ be closed and, additionally, a $(d-1)$-set. 
 \begin{enumerate} 
 \item There is a continuous restriction operator $\mathcal R_F$ which maps
	every space $W^{1,p}(\R^d)$ continuously onto the Besov space
	$B^{1-\frac {1}{p}}_{p,p}(F)$ as long as $p \in {]1,\infty[}$.
 \item Conversely, there is an extension operator $\mathcal E_F$ which maps each
	space $B^{1-\frac {1}{p}}_{p,p}(F)$ continuously into $W^{1,p}(\R^d)$,
	provided $p \in {]1,\infty[}$.
 \item By construction, $\mathcal E_F$ is a right inverse for $\mathcal R_F$,
	i.e.\@ $\mathcal R_F \mathcal E_F$ is the identity operator on
	$B^{1-\frac {1}{p}}_{p,p}(F)$, cf.\@ \cite[Ch.V.1.3]{jons}.
 \end{enumerate}
\end{proposition}

It turns out that the extension operator $\mathcal E_F$ even maintains
Lipschitz continuity:

\begin{theorem} \label{t-Jons}
 Let $F \subset \R^d$ be closed and, additionally, a $(d-1)$-set. Then the
 operator $\mathcal E_F$ from Proposition~\ref{p-JW} maps the space of
 Lipschitz continuous functions on $F$ continuously into the space of
 Lipschitz continuous functions on $\R^d$.
\end{theorem}

\begin{proof}
 The extension operator $\mathcal E_F$ is of Whitney type, and we need some
 facts about the Whitney decomposition of $\mathbb R^d\setminus F$ and a
 related partition of unity $\{\phi_i\}$, cf. \cite{jons} for more background
 and details. The decomposition is a collection of closed, dyadic cubes $Q_i$,
 with sidelength $2^{N_i}$ for integers $N_i$, and with mutually disjoint
 interiors, such that $\bigcup Q_i = \mathbb R^d\setminus F$, and 
 \begin{equation}\label {aj-1}
  {\rm diam} Q_i \leq \mathrm{d} (Q_i,F) \leq 4 \rm {diam} Q_i,
 \end{equation}
 where $\mathrm{d}(Q_i,F)$ is the distance between $Q_i$ and $F$. Denote the
 diameter of $Q_i$ by $l_i$, its sidelength by $s_i$,  and its center by
 $\mathrm x_i$, and let $Q_i^\star$ denote the cube obtained by expanding $Q_i$
 around its center with a factor $\iota$, $1<\iota<5/4$. It follows from
 \eqref{aj-1} that
 \begin{equation}\label {aj-2}
  1/4 l_i\leq l_k \leq 4l_i,
 \end{equation}
 if $Q_i$  and $Q_k$ touch. This means that  $Q_i^\star$ intersects a cube $Q_k$
 only if $Q_i$ and $Q_k$ touch, and that each point in $\mathbb R^d\setminus F$ is
 contained in at most $N_0$ cubes $Q_i^\star$, where $N_0$ is a number depending
 only on the dimension $d$.

 Next, nonnegative $C^\infty$-functions $\phi_i$ are chosen in such a way that
 $\phi_i(\mathrm x) = 0$ if $\mathrm x \notin Q_i^\star$, $\sum_i
 \phi_i(\mathrm x) = 1$, $\mathrm x \in \mathbb R^d\setminus F$, and so that
 $|D^j \phi_i| \leq c l_i^{-|j|}$ for any $j$, where $c$ depends on $j$. 

 Let $I$ denote those $i$ such that $s_i\leq 1$, let $\rho$ again be the
 restriction of the ($d-1$)-dimensional Hausdorff measure on $F$, and put $c_i =
 \rho (B(\mathrm x_i, 6l_i))^{-1}$. Note that it follows from \eqref{e-0} and
 \eqref{aj-1}, that $\rho (B(\mathrm x_i, 6l_i)) > 0$. The extension operator used
 in Proposition~\ref{p-JW} is given by
 \begin{equation}\label {aj-3}
  {\mathcal E}_F f(\mathrm x) = \sum_{i \in I} \phi_i(\mathrm x) c_i
	\int_{|\mathrm t - \mathrm x_i| \le 6 l_i} f(\mathrm t) \; \mathrm d \rho
	(\mathrm t), \quad \mathrm x \in \R^d \setminus F,
 \end{equation}
 and ${\mathcal E}_F f(\mathrm x) = f(\mathrm x)$ for $\mathrm x \in F$.

 We now head for Lipschitz continuity of $\mathcal E_F f$. To begin with let
 $\mathrm x$ and $\mathrm y$ be in cubes with sides $\leq 1/4$. Then $\sum
 \phi_i(\mathrm x) = \sum \phi_k(\mathrm y) = 1$, where the sums are taken over
 all $i$ and $k$, respectively. Using this, one obtains, for any constant $b$,
 \begin{equation}\label {aj-4}
  {\mathcal E}_F f(\mathrm x) - b = \sum_i \phi_i(\mathrm x) c_i
	\int_{|\mathrm t - \mathrm x_i| \le 6l_i} ( f(\mathrm t) - b )
	\; \mathrm d \rho (\mathrm t),
 \end{equation}
 and taking $b = {\mathcal E}_F f(\mathrm y)$
 \begin{equation}\label {aj-5}
  {\mathcal E}_F f(\mathrm x) - {\mathcal E_F f}(\mathrm y) = \sum_i \sum_k
	\phi_i(\mathrm x) \phi_k(\mathrm y) c_i c_k \int
	\int_{|\mathrm t - \mathrm x_i| \le 6 l_i,
	|\mathrm s - \mathrm x_k| \le 6l_k} (f(\mathrm t) - f(\mathrm s))
	\; \mathrm d \rho (\mathrm t) \mathrm d \rho(\mathrm s).
 \end{equation}
 We also have
 \begin{equation} \label{aj-6}
  D^j ({\mathcal E}_F f)(\mathrm x) = \sum_i D^j \phi_i (\mathrm x) c_i
	\int_{|\mathrm t - \mathrm x_i| \le 6l_i} f(\mathrm t) \; \mathrm d \rho
	(\mathrm t),
 \end{equation}
 and, for $|j|>0$, since then $\sum_i D^j \phi_i(\mathrm x) = 0$, so we can
 subtract ${\mathcal E}_F f(\mathrm y)$ from the integrand,
\begin{equation}\label {aj-7}
  D^j ({\mathcal E}_F f) (\mathrm x) = \sum_i \sum_k D^j \phi_i(\mathrm x)
	\phi_k(\mathrm y) c_i c_k \int\int_{|\mathrm t - \mathrm x_i| \le 6l_i,
	|\mathrm s - \mathrm x_k| \le 6l_k} (f(\mathrm t) - f(\mathrm s))
	\; \mathrm d \rho(\mathrm t) \mathrm d \rho(\mathrm s).
 \end{equation}

 Assume now that $f$ is Lipschitz continuous with Lipschitz norm $1$. Let
 $\mathrm x \in Q_\nu$, $\mathrm y \in Q_\eta$, where, say, $s_\nu \geq
 s_\eta$, and assume first $s_{\nu} \leq 1/4$. If $|\mathrm x - \mathrm y| <
 l_\nu/2$, then by the mean value theorem $\mathcal E_F f(\mathrm x) -
 \mathcal E_F f(\mathrm y) = \nabla (\mathcal E_F f)(\xi) \cdot (\mathrm x -
 \mathrm y)$ for some $\xi$ with $|\mathrm x -
 \xi| < l_\nu/2$. Note that the geometric constellation assures that the whole
 segment joining $\mathrm x$ and $\mathrm y$ avoids $F$, so the mean value
 theorem is applicable. Next, take $\kappa$ so that  $\xi \in Q_\kappa$. Now we
 use, if $s_{\kappa}\leq 1/4$ (otherwise, see below), \eqref{aj-7} with
 $\mathrm x$ and $\mathrm y$ equal to $\xi$, and recall that if $\phi_i(\xi )
 \neq 0$, then $Q_i$ and $Q_\kappa$ touch. For nonzero terms we then have, for
 $\mathrm t$ and $\mathrm s$ in the domain of integration, $|\mathrm t -
 \mathrm s| \le |\mathrm t - \mathrm x_i| + |\mathrm x_i - \mathrm x_\kappa| +
 |\mathrm x_\kappa - \mathrm x_k| + |\mathrm x_k - s| \le 7 l_i + 2 l_\kappa +
 7 l_k$, and also, by \eqref{aj-2}, that, $l_i$ and $l_k$ are comparable to
 $l_\kappa$. Recalling that $0 \le \phi_i \le 1$, $|D^j \phi_i| \le c l_i^{-1} $
 for $|j| = 1$ and using $|f(\mathrm s) - f(\mathrm t)| \le
 |\mathrm t-\mathrm s|$, one immediately obtains $|D^j ({\mathcal E}_F f)(\xi)|
 \le c$ for $|j| = 1$, so
 \begin{equation} \label{aj-8}
  |{\mathcal E}_F f(\mathrm x) - {\mathcal E}_F f(\mathrm y)| \le c |\mathrm x
	- \mathrm y|.
 \end{equation}
 If $|\mathrm x - \mathrm y| \ge l_\nu/2$, we use \eqref{aj-5} together with the
 observation that now $|\mathrm t - \mathrm s| \le |\mathrm t - \mathrm x_i| +
 |\mathrm x_i - \mathrm x| + |\mathrm x - \mathrm y| + |\mathrm y - \mathrm y_k| +
 |\mathrm y_k - \mathrm s| \le 7 l_i + l_\nu + l_\eta + 7 l_k + |\mathrm x -
 \mathrm y| \le 58 l_\nu + |\mathrm x - \mathrm y| \le c |\mathrm x - \mathrm y|$
 if $\phi(\mathrm x)$ and $\phi(\mathrm y)$ are nonzero, and obtain again
 \eqref{aj-8}. If instead $\mathrm y \in F$ we get the same result using
 \eqref{aj-4} with $b = f(\mathrm y)$ and $|\mathrm t - \mathrm y| \le
 |\mathrm t - \mathrm x_i| + |\mathrm x_i - \mathrm x| + |\mathrm x - \mathrm y|
 \le 7 l_i + l_\nu + |\mathrm x - \mathrm y| \le c |\mathrm x - \mathrm y|$,
 since, by \eqref{aj-1}, $l_\nu \le |\mathrm x - \mathrm y|$.

 If $s_\nu > 1/4$, or $s_\kappa > 1/4$, we can no longer use \eqref{aj-7},
 \eqref{aj-5}, and \eqref{aj-4}. In the case $|\mathrm x - \mathrm y| < l_\nu/2$,
 \eqref{aj-6} together with  $|f| \le 1$ gives the desired estimate
 $|D^j ({\mathcal E}_F f)(\xi)| \le c l_\kappa^{-1} \le c$ for $|j| = 1$. Using
 \eqref{aj-3} we see  that $|{\mathcal E}_F f| \leq c$ everywhere, which in
 particular implies \eqref{aj-8} in the remaining cases.
\end{proof}

\begin{rem} \label{r-wichtig}
 \begin{enumerate}
  \item Since the detailed structure of the Besov spaces
	$B^{1-\frac {1}{p}}_{p,p}(F)$ is not of interest in this paper, we
	refer to \cite[Ch.~V.1]{jons} for a definition.
  \item It is known that, for any $f \in W^{1,p}(\R^d)$,
	\begin{equation} \label{e-pointwi}
	  \lim_{r \to 0} \frac {1}{|B(\mathrm y,r)|}\int_{B(\mathrm y,r)}
		f(\mathrm x) \; d \mathrm x
	\end{equation}
	exists for $\mathcal H_{d-1}$-almost all $\mathrm y \in \R^d$ (even more
	is true, see \cite[Ch.~3.1]{ziemer}). Moreover, the function, defined by
	\eqref{e-pointwi}, reproduces $f$ within its Sobolev class, and the
	restriction of $f$ to any $(d-1)$-set $F$ is established this way, compare
	\cite[Ch.~2.1]{jons}.
  \item The proof of Theorem~\ref{t-Jons} does in fact not require much about the
	measure $\rho$. The only thing needed is that the measure of any ball with
	center in $F$ is positive, which in particular holds for any $l$-measure
	with $0 < l \le n$.
 \end{enumerate}
\end{rem}
\noindent
For all what follows we fix an open ball $B$ which contains $\overline \Omega$.
In the sequel we consider in our case $F = D$ the restriction/extension operators
$\mathcal R_F/\mathcal E_F$ not only on all of $\R^d$, but also on the ball $B$.
Since $D \subset B$ and the restriction operator $\mathcal R_D$ takes into account
only the local behaviour of functions near $D$, the operator $\mathcal E_D$
remains a right inverse of $\mathcal R_D$ in this understanding. In this spirit,
we also maintain the notations $\mathcal E_D, \mathcal R_D$.
\begin{definition} \label{d-modspace}
 If $\Lambda \subseteq \R^d$ is a domain and $F \subset \Lambda $ is a
 $(d-1)$-set, then we write
 \[ \mathcal W^{1,p}_F(\Lambda) := \bigl\{ \psi \in W^{1,p}(\Lambda) :
	\mathcal R_F \psi=0 \; \text{a.e.} \;\text {on} \; F \bigr\},
 \]
 where the measure on $F$ is again $\mathcal H_{d-1}|_F$,
 cf.\@ Definition~\ref{d-lset}.
\end{definition}
It is a natural question whether $\mathcal W^{1,p}_F(\Lambda) =
W^{1,p}_F(\Lambda)$ holds. An affirmative answer will be given in
Corollary~\ref{c-coincid} -- which will serve as a technical tool for the proof
of the interpolation results below.

\section{Interpolation} \label{sec3}

\noindent
In this section we establish interpolation results that are well-known for 
$\R^d$ or smooth domains, for the spaces $W^{1,p}_D(\Omega)$. As already
mentioned in the introduction, the crucial ingredient is a Sobolev extension
operator. So we introduce the following assumption.

\begin{assu} \label{a-extendgeneral}
 There exists a linear, continuous extension operator $\mathfrak E :
 W^{1,1}_D(\Omega) \to W^{1,1}_D(\R^d)$ which simultaneously defines a continuous
 extension operator $\mathfrak E : W^{1,p}_D(\Omega) \to W^{1,p}_D(\R^d)$ for
 every $p \in {]1,\infty[}$.
\end{assu}

\begin{rem} \label{r-remain}
 \begin{enumerate}
  \item \label{r-remain:i} We are aware that Assumption~\ref{a-extendgeneral}
	is of quite different character in comparison to Assumption~\ref{a-1}.
	Only by formulating the results in this abstract way, it becomes
	manifest that it is only the functorial property of the extension
	operator that carries over the interpolation results.
	 However, in Section~\ref{extens} we will subsequently
	establish geometric conditions on $\Omega$ and $D$ that will assure
	Assumption~\ref{a-extendgeneral}.
  \item \label{r-remain:ii} Combining the mapping $\mathfrak E$ with the
	operator that restricts any function on $\R^d$ to $B$, one obtains an
	operator that maps $W^{1,p}_D(\Omega)$ continuously into the space
	$W^{1,p}_D(B)$; we maintain the notation $\mathfrak E$ for the
	resulting operator.
  \item \label{r-remain:iii} Under Assumptoin \ref{a-extendgeneral}, one can 
        establish the corresponding Sobolev embeddings $W^{1,p}_D(\Omega) \to
        L^q(\Omega)$ (compactness, included) in a straightforward manner.
 \end{enumerate}
\end{rem}

Our main result on interpolation is the following.

\begin{theorem} \label{t-1interpolgen}
 Let Assumptions~\ref{a-1} and~\ref{a-extendgeneral} be satisfied. Then complex
 and real interpolation between the spaces of the family $\{ W^{1,p}_D(\Omega)
 \}_{p \in {]1,\infty[}}$ act as for the family $\{ W^{1,p}(\R^d)
 \}_{p \in {]1,\infty[}}$. In particular, one has for $p_0, p_1 \in
 {]1,\infty[}$ and $\frac 1p = \frac{1 - \theta}{p_0} + \frac{\theta}{p_1}$
 \[ \bigl[ W^{1,p_0}_D(\Omega), W^{1,p_1}_D(\Omega) \bigr]_\theta =
	W^{1,p}_D(\Omega) = \bigl( W^{1,p_0}_D(\Omega), W^{1,p_1}_D(\Omega)
	\bigr)_{\theta,p},
 \]
\end{theorem}

\begin{corollary} \label{c-remaintrue-}
 Let $\hat W^{-1,q}_D(\Omega)$ denote the dual of $W^{1,q'}_D(\Omega)$,
 $\frac{1}{q} + \frac{1}{q'} = 1$ and $W^{-1,q}_D(\Omega)$ denote the space of
 continuous antilinear forms on $W^{1,q'}_D(\Omega)$, $\frac{1}{q} +
 \frac{1}{q'} = 1$. For $p_0, p_1 \in {]1,\infty[}$ and $\frac{1}{p} =
 \frac{1 - \theta}{p_0} + \frac{\theta}{p_1}$ one has
 \begin{equation} \label{e-interpol1-}
  \bigl[ \hat W^{-1,p_0}_D(\Omega), \hat W^{-1,p_1}_D(\Omega) \bigr]_\theta =
	\hat W^{-1,p}_D(\Omega),
 \end{equation}
 and 
 \begin{equation} \label{e-interpol1-1}
  \bigl[ W^{-1,p_0}_D(\Omega), W^{-1,p_1}_D(\Omega) \bigr]_\theta =
	W^{-1,p}_D(\Omega).
 \end{equation}
\end{corollary}

\begin{proof}
 Concerning \eqref{e-interpol1-}, one employs the duality formula for complex
 interpolation in case of reflexive Banach spaces (see
 \cite[Ch.~1.11.3]{triebel}), which reads as $[X',Y']_\theta = [X,Y]'_\theta$.
 In order to get \eqref{e-interpol1-1}, one associates to any linear form
 $T$ an antilinear form $T_a$ defining $\langle T_a, \psi \rangle := \langle
 T, \overline \psi \rangle$. It is clear that the mappings $T \mapsto T_a$ and
 $T_a \mapsto T$ form a retraction/coretraction pair, thus
 \eqref{e-interpol1-1} may be derived from \eqref{e-interpol1-} by the
 retraction/coretraction theorem for interpolation.
\end{proof}

Theorem~\ref{t-1interpolgen} will be proved in two steps. First we establish
the corresponding result for the spaces $W^{1,p}_F(B)$ where $B$ is a ball and
$F \subseteq B$ is a $(d-1)$-set. From this we will then deduce the general
statement.

One main ingredient is the Jonsson/Wallin result from Proposition~\ref{p-JW}
and Theorem~\ref{t-Jons}. We use this in the following way: the right inverse
property of $\mathcal E_F$ for $\mathcal R_F$ implies that $\mathcal E_F
\mathcal R_F : W^{1,p}(B) \to  W^{1,p}(B)$ is a (continuous) projection.
Furthermore, it is straightforward to verify that $\mathcal E_F \mathcal R_F
\varphi = 0$, iff $\mathcal R_F \varphi=0$. This implies that $\varphi \in
\mathcal W_F^{1,p}(B)$, if and only if $\varphi \in W^{1,p}(B)$ and
$(1 - \mathcal E_F \mathcal R_F) \varphi = \varphi$. Consequently, the operator
$\mathcal P := 1 - \mathcal E_F \mathcal R_F$ is a (continuous) projection from
$W^{1,p}(B)$ onto $\mathcal W_F^{1,p}(B)$.

The existence of the projector $\mathcal P$ allows to deduce the desired
interpolation properties for the spaces $\mathcal W^{1,p}_F(\Omega)$ by purely
functorial properties.

\begin{theorem} \label{t-1interpol}
 Let $F \subset  B$ be a $(d-1)$-set. Then the spaces $\mathcal W^{1,p}_F(B)$ 
$(p \in {]1,\infty[})$ interpolate according to the same rules as the
 spaces $W^{1,p}(B)$ do. This affects any interpolation functor, in particular
real and complex interpolation.
\end{theorem}

\begin{proof}
 Let $\mathcal P$ be the projection from above. Since, for any $p \in
 {]1,\infty[}$, $\mathcal P$ maps $W^{1,p}(B)$ onto $\mathcal W_F^{1,p}(B)$,
 interpolation carries over from the spaces $W^{1,p}(B)$ to the spaces
 $\mathcal W_F^{1,p}(B)$ by a classical interpolation principle for
 complemented subspaces, see \cite[Ch.~1.17.1]{triebel}.
\end{proof}

In order to obtain this also for the spaces $W^{1,p}_D(\Omega)$, we will prove
the following

\begin{theorem} \label{t-coincid}
 Let $F \subset \R^d$ be a $(d-1)$-set. Then the spaces $ W^{1,p}_F(\R^d)$ and 
$\mathcal W^{1,p}_F(\R^d)$ in fact coincide for $p \in {]1,\infty[}$.
\end{theorem}

\begin{proof}
 The inclusion $W^{1,p}_F(\R^d) \subseteq \mathcal W^{1,p}_F(\R^d)$ is implied
 by the Jonsson/Wallin result: all functions $\psi$ from $C^\infty_F(\R^d)$
 vanish in a neighbourhood of $F$ and, hence, have trace $0$ on $F$,
 i.e.\@ $\mathcal R_F \psi = 0$. Since the trace is a continuous operator into
 $L^1(F;\rho)$, this remains true for all elements from $W^{1,p}_F(\R^d)$.

 Conversely, assume $\psi \in \mathcal W^{1,p}_F(\R^d)$. By the definition of
 the projector $\mathcal P = 1 - \mathcal E_F \mathcal R_F$ one has $\mathcal P
 \psi = \psi$. Since $ \psi \in \mathcal W_F^{1,p}(\R^d) \subset
 W^{1,p}(\R^d)$, there is a sequence $\{\psi _k\}_k$ from $C_0^\infty(\R^d)$
 that converges towards $ \psi$ in the $W^{1,p}(\R^d)$ topology. Clearly, then
 $\mathcal P \psi_k \to \mathcal P \psi = \psi$, and the elements $\mathcal P
 \psi_k$ fulfill, by the definition of $\mathcal P$, the condition $\mathcal P
 \psi_k = 0$ a.e.\@ on $F$ with respect to $\rho$. Thus we have $\mathcal R_F
 (\mathcal P \psi_k) = 0$.

 We fix $k$ and denote $\mathcal P \psi_k$ by $f$ for brevity. Our intention is
 to show:
 \begin{equation} \label{Sternchen}
  \text{There exists } g \in C^\infty(\R^d) \text{ with } \supp(g) \cap F =
	\emptyset \text{ and } \| f - g \|_{W^{1,p}(\R^d)} \le \frac {1}{k}.
 \end{equation}
 
 By the construction of the projector $\mathcal P = 1 - \mathcal E_F
 \mathcal R_F$ and the Jonsson/Wallin results in Proposition~\ref{p-JW} the
 function $f$ is Lipschitzian and vanishes almost everywhere on $F$. We will
 now show that, in fact, it vanishes identically on $F$. Let $\mathrm x \in F$
 be an arbitrary point. Then, for every $r > 0$, one has $\rho(F \cap
 B(\mathrm x,r)) > 0$ because $F$ is a $(d-1)$-set. Thus, in this ball there is
 a point $\mathrm y \in F$ for which $f(\mathrm y) = 0$ holds. Hence,
 $\mathrm x$ is an accumulation point of the set on which $f$ vanishes, and the
 claim follows from the continuity of $f$.

 Let now $\{\zeta_n\}_n$ be the sequence of cut-off functions, defined on
 $\R_+$ by
 \[ \zeta_n(t) = \begin{cases}
		0, & \text{if } 0  \le t  \le 1/n, \\
		nt - 1, & \text{if } 1/n  \le t  \le 2/n, \\
		1,  & \text{if } 2/n < t.
	\end{cases}
 \]
 Note that for $t \neq 0$ the values $\zeta_n(t)$ tend to $1$ as $n \to
 \infty$. Moreover, one has $0 \le  t \zeta'_n(t) \le 2$ and  $t \zeta'_n(t)$
 tends to $0$ for all $t$. We denote by $\dist_F : \R^d \to \R_+$ the function
 which measures the distance to the set $F$. Note that $\dist_F$ is
 Lipschitzian with Lipschitz constant $1$. Hence, it is a.e.\@ differentiable
 with $|\nabla \dist_F| \le 1$, see \cite[Ch.~4.2.3]{ev/gar}. Define $w_n :=
 \zeta_n \circ \dist_F$. Note that $w_n \to 1$ almost everywhere in $\R^d$ when
 $n \to \infty$. Moreover, since $\zeta_n$ is piecewise smooth, one calculates,
 according to the chain rule (see \cite[Ch.~7.4]{gil}),
 \[ \nabla w_n(x) = \begin{cases}
		\zeta'_n (\dist_F(\mathrm{x})) \nabla \dist_F(\mathrm{x}), & 
			\text{if } \dist_F(\mathrm{x}) \in \bigl] \frac 1n,
			\frac 2n \bigr[, \\
		0, & \text{else}.
	\end{cases}
 \]
 Since $|\nabla \dist_F| \le 1$ a.e., $\dist_F \nabla w_n$ is uniformly (in
 $n$) bounded a.e.\@ and converges a.e.\@ to $0$ as $n \to \infty$.  Let $f_n
 = f w_n$. We claim that $f_n - f = f (1 - w_n) \to 0$ in $W^{1,p}(\R^d)$. By
 the dominated convergence theorem, $f ( 1 - w_n ) \to 0$ in $L^p(\R^d)$ since
 $w_n \to 1$. Now, for the gradient holds
 \[ \nabla (f_n - f) = (1 - w_n) \nabla f + f \nabla w_n \quad
	\text{a.e.\@ on } \R^d.
 \]
 Again by the dominated convergence theorem, the first term converges to $0$
 in $L^p(\R^d)$. It remains to prove that $\|f \nabla w_n\|_{L^p} \to 0$. We
 have
 \begin{equation} \label{e-lpnorm}
  \| f \nabla w_n \|^p_{L^p} = \int_{\R^d} \Bigl| \frac {f}{\dist_F} \Bigr|^p
	\Bigl| \dist_F \nabla w_n \Bigr|^p \mathrm{d} \mathrm{x}.
 \end{equation}
 Due to the fact that $f$ vanishes identically on $F$ and the Lipschitz
 property of $f$, the function $\frac {f}{\dist_F}$ is bounded. Hence, again
 dominated convergence yields $f\nabla w_n \to 0$ in $L^p(\R^d)$. The support
 of each function $f_n$ has a positive distance to the set $F$. Thus, it
 suffices to convolve a function $f_n$ (according to a sufficiently high index
 $n$) with a smooth mollifying function with small support to obtain $g$, which
 proves \eqref{Sternchen}.
Finally, the assertion follows from a $3\epsilon$ argument.
\end{proof}
\begin{corollary} \label{c-coincid}
 Let $B \subset \R^d$ be an open ball and $F \subset B$ be a $(d-1)$-set. Then
 the spaces $ W^{1,p}_F(B)$ and $\mathcal W^{1,p}_F(B)$ in fact coincide for $p
 \in {]1,\infty[}$.
\end{corollary}
\begin{proof}
The inclusion $ W^{1,p}_F(B) \subseteq \mathcal W^{1,p}_F(B)$ is clear. Conversely,
 let, for any function $\psi \in \mathcal W^{1,p}_F(B)$, $\widehat \psi$ be a
 $W^{1,p}(\R^d)$-extension. Since $F$ is contained in $B$ we still have
 $\widehat \psi \in \mathcal W^{1,p}_F(\R^d)$. Hence, due to
 Theorem~\ref{t-coincid}, the function $\widehat \psi$ may be approximated in
 the $W^{1,p}(\R^d)$-norm by a sequence $\{\psi_k\}_k$ from $C^\infty_F(\R^d)$.
 Evidently, then $\psi$ may be approximated by the sequence $\{\psi_k|_B\}_k$
 in the $W^{1,p}(B)$-norm.
\end{proof}
\begin{rem} \label{r-cutoff}
 \begin{enumerate} 
 \item The basic idea for the proof of Theorem~\ref{t-coincid} is analogous to
	that in \cite[Prop.~3.12]{jer/ke}.
 \item Seemingly, the coincidence of the spaces $W^{1,p}_F(B)$ and $\mathcal
	W^{1,p}_F(B)$ is only of limited, more technical interest. This,
	however, is not the case: on the one hand it is often considerably
	simpler to prove that a certain function belongs to the space
	$\mathcal W^{1,p}_F$, see the proof of Theorem~\ref{l-extension} below,
        compare also \cite[Ch.~VIII.1]{jons} or
	\cite[Ch.~6.6]{mazyasob}. On the other hand, it is of course often
	more comfortable, if one has to prove a certain property for all
	elements from $\mathcal W^{1,p}_F(B)$ and may confine oneself, by
	density, to the functions from $C^\infty_F(B)$.
\item Theorem~\ref{t-coincid} heavily rests on the property of $F$ to be a
	$(d-1)$-set: suppose e.g.\@ $p>d$ and assume that $\mathrm{x} \in F$
	is an isolated point. Then, for every $\psi \in C^\infty_F(\Omega)$
	one has $\psi(\mathrm{x}) = 0$, what clearly extends to all $\psi \in
	W^{1,p}_F(\Omega)$, since the Dirac measure $\delta_\mathrm{x}$ is a
	continuous linear form on $W^{1,p}(\Omega)$. On the other hand, the
	condition $\mathcal R_F \psi = 0$ a.e.\@ on $F$ does not impose a
	condition on $\psi$ in the point $\mathrm{x}$ because $\{ \mathrm{x}
	\}$ is of measure $0$ with respect to $\rho=\mathcal H_{d-1}|_F$.
\end{enumerate}
\end{rem}

\begin{corollary} \label{c-remaintrue}
 Concerning real and complex interpolation, Theorem~\ref{t-1interpol} remains
 true, if there $\mathcal W^{1,p}_F(B)$ is replaced by $W^{1,p}_F(B)$. In
 particular, one has for $p_0, p_1 \in {]1,\infty[}$ and $\frac 1p =
 \frac{1 - \theta}{p_0} + \frac{\theta}{p_1}$
 \[ \bigl[ W^{1,p_0}_F(B), W^{1,p_1}_F(B) \bigr]_\theta = W^{1,p}_F(B) =
	\bigl( W^{1,p_0}_F(B), W^{1,p_1}_F(B) \bigr)_{\theta,p}.
 \]
%compare \cite[Ch.~2.4.2]{triebel}.
\end{corollary}

\begin{proof}
 The assertion concerning complex interpolation is immediate from
 Theorem~\ref{t-1interpol} and Theorem~\ref{t-coincid}, which also imply the
 right equality. Considering real interpolation, one gets
 \begin{equation} \label{e-intspace}
  \bigl( W^{1,p_0}_F(B), W^{1,p_1}_F(B) \bigr)_{\theta,q} = \bigl(
	\mathcal W^{1,p_0}_F(B), \mathcal W^{1,p_1}_F(B) \bigr)_{\theta,q}.
 \end{equation}
 According to Theorem~\ref{t-1interpol}, the right hand side is some Besov
 space (see \cite[Ch.~2.4.2]{triebel}) including again the trace-zero
 condition on $F$. It is clear that $C^\infty_F(B)$ is contained in this
 space. What remains to show is that $C^\infty_F(B)$ is also dense in this
 space.

 Let us suppose, without loss of generality, $p_1 > p_0$. By definition,
 $C^\infty_F(B)$ is dense in $W^{1,p_1}_F(B)$ with respect to its natural
 topology. Moreover, $W^{1,p_1}_F(B)$ is dense in the interpolation space
 \eqref{e-intspace} (see \cite[Ch.~1.6.2]{triebel}), and the topology of
 this interpolation space is weaker than that of $W^{1,p_1}_F(B)$. Hence,
 $C^\infty_F(B)$ is indeed dense in the corresponding interpolation space, or,
 in other words: the interpolation space is the closure of $C^\infty_F(B)$
 with respect to the corresponding Besov topology.
\end{proof}
\begin{rem} \label{r-excluded}
 Concerning real interpolation, the interpolation indices $(\theta,\infty)$
 have indeed to be excluded, compare \cite[Ch.~1.6.2]{triebel}. The crucial
 point is that the smaller space has to remain dense in the corresponding
 interpolation space.
\end{rem}

We now turn to the proof of Theorem~\ref{t-1interpolgen}. We first introduce
the following definition.

\begin{definition} \label{d-ext/res}
 We denote by $\mathfrak R : W^{1,p}(B) \to W^{1,p}(\Omega)$ the canonical
 restriction operator.
\end{definition}

\begin{rem} \label{r-beschrgeb}
 It is not hard to see that the canonical restriction operator $\mathfrak R :
 W^{1,p}(B) \to W^{1,p}(\Omega)$ gives rise to a restriction operator
 $\mathfrak R : W_D^{1,p}(B) \to W^{1,p}_D(\Omega)$ -- for which we also
 maintain the notation $\mathfrak R$. Note that $\mathfrak E$ and $\mathfrak R$
 are consistent on the sets $\{W_D^{1,p}(B)\}_{p \in [1,\infty[}$ and
 $\{W_D^{1,p}(\Omega)\}_{p \in [1,\infty[}$: if $q > p$, then $\mathfrak R :
 W_D^{1,q}(B) \to W_D^{1,q}(\Omega)$ is the restriction of $\mathfrak R :
 W_D^{1,p}(B) \to W_D^{1,p}(\Omega)$ and $\mathfrak E : W_D^{1,q}(\Omega) \to
 W_D^{1,q}(B)$ is the restriction of $\mathfrak E : W_D^{1,p}(\Omega) \to
 W_D^{1,p}(B)$.

 Finally, one observes that, for every $p \in [1,\infty[$, the operators
 $\mathfrak R : W_D^{1,p}(B) \to W^{1,p}_D(\Omega)$ and $\mathfrak E:
 W_D^{1,p}(\Omega) \to W^{1,p}_D(B)$ form a retraction/coretraction pair, see
 \cite[Ch.~1.2.4]{triebel}.
\end{rem}

%\begin{rem}
% This property to be a $(d-1)$-set is often also called \emph{Ahlfors-David
% regularity condition}. cf. \cite{}\footnote{Zitat?}
%\end{rem}

\begin{proof}[Proof of Theorem~\ref{t-1interpolgen}]
 Let $B \supset \overline \Omega$ be the ball introduced above. Firstly,
 Corollary~\ref{c-remaintrue} shows how the spaces from the family
 $\{W^{1,p}_D(B)\}_{p \in {]1,\infty[}}$ interpolate. Secondly, the
 extension/restriction operators $\mathfrak E/ \mathfrak R$ (compare
 Remark~\ref{r-remain}~\ref{r-remain:ii})
 together with the retraction/coretraction theorem, see
 \cite[Ch.~1.2.4]{triebel}, allow to carry over interpolation between spaces
 from $\{ W^{1,p}_D(B) \}_{p \in {]1,\infty[}}$ to the spaces from
 $\{ W^{1,p}_D(\Omega) \}_{p \in {]1,\infty[}}$.
\end{proof}

\section{The extension operator} \label{extens}
As already mentioned in Remark~\ref{r-remain}, the condition of the
extendability for $W^{1,p}_D(\Omega)$ within the same class is an abstract one
which should be supported by geometric conditions on $\Omega$ and on $D$. We
will do this within this section. In a first step we will establish three
general principles. First, we open the possibility of passing from the domain
$\Omega$ to another domain $\Omega_\bullet$ with a reduced Dirichlet boundary
part, while $\Gamma = \partial \Omega \setminus D$ remains part of
$\partial \Omega_\bullet$. In most cases this improves the boundary geometry in
view of the $W^{1,p}$-extendability, see the example in
Figure~\ref{fig-Meissel} above. Secondly, we show that only the local geometry
of the domain around the boundary part $\Gamma$ plays a role for the existence
of such an extension operator. Thirdly, we prove that -- under very general
geometric assumptions -- the extended functions do have the adequate trace behavior on
$D$ for \emph{every} extension operator.

In the second subsection we then give conditions for geometries around the
boundary part $\Gamma$, which do -- together with the results from\
Subsection~\ref{subsec:Sob1} -- really imply the validity of
Assumption~\ref{a-extendgeneral}.

\subsection{Sobolev extension: general features} \label{subsec:Sob1}
The first point we address is the following: as in Figure~\ref{fig-Meissel}
there may be boundary parts which carry a Dirichlet condition and belong to the
inner of the closure of the domain under consideration. Then one can extend the
functions on $\Omega$ by $0$ to such boundary part, thereby enlarging the
domain and simplifying the boundary geometry. In the following we make this
precise.
\begin{lemma} \label{l-modigeb}
 Let $\Omega \subset \R^d$ be a domain and let $E \subset \partial \Omega$ be
 compact. Define $\Omega_\bullet$ as the interior of the set $\Omega \cup E$.
 Then the following holds true.
 \begin{enumerate} 
  \item The set $\Omega_\bullet$ is again a domain, $\Xi := \partial \Omega
	\setminus E$ is a (relatively) open subset of $\partial \Omega_\bullet$
	and $\partial \Omega_\bullet = \Xi \cup (E \cap \partial
	\Omega_\bullet)$.
 \item Extending functions from $W^{1,p}_E(\Omega)$ by $0$ to $\Omega_\bullet$,
	one obtains an isometric extension operator
	$\mathrm{Ext}(\Omega,\Omega_\bullet)$ from $W^{1,p}_E(\Omega)$ onto
	$W^{1,p}_E(\Omega_\bullet)$.
 \end{enumerate}
\end{lemma} 
\begin{proof}
 \begin{enumerate}
  \item Due to the connectednes of $\Omega$ and the set inclusion $\Omega \subset
	\Omega_\bullet \subset \overline \Omega$, the set $\Omega_\bullet$ is
	also connected, and, hence, a domain. Obviously, one has
	$\overline{\Omega_\bullet} = \overline \Omega$. This, together with the
	inclusion $\Omega \subset \Omega_\bullet$ leads to
	$\partial \Omega_\bullet \subset \partial \Omega$. Since $\Xi \cap
	\Omega_\bullet = \emptyset$, one gets $\Xi \subset
	\partial \Omega_\bullet$. Furthermore, $\Xi$ was relatively open in
	$\partial \Omega$, the more it is relatively open in
	$\partial \Omega_\bullet$.

	The last asserted equality follows from $\partial \Omega_\bullet = (
	\Xi \cap \partial \Omega_\bullet) \cup (E \cap
	\partial \Omega_\bullet)$ and $\Xi \subset \partial \Omega_\bullet$.
  \item Consider any $\psi \in C^\infty_E(\R^d)$ and its restriction
	$\psi|_\Omega$ to $\Omega$. Since the support of $\psi$ has a positive
	distance to $E$, one may extend $\psi|_\Omega$ by $0$ to the whole of
	$\Omega_\bullet$ without destroying the $C^\infty$--property.
        Thus, this extension operator
	provides a linear isometry from $C^\infty_E(\Omega)$ \emph{onto}
	$C^\infty_E(\Omega_\bullet)$ (if both are equipped with the 
	$W^{1,p}$-norm). This extends to a linear extension operator
	$\mathrm{Ext}(\Omega,\Omega_\bullet)$ from $W^{1,p}_E(\Omega)$ onto
	$W^{1,p}_E(\Omega_\bullet)$, see the two following commutative
	diagrams:

 \begin{equation*}
 \begin{tikzcd}[row sep=large, column sep=huge]
 	C^{\infty}_E (\mathbb{R}^d)
 	\arrow{r}{\operatorname*{restrict}_{\mathbb{R}^d\to\Omega}} 
 	\arrow{d}[swap]{\operatorname*{restrict}_{\mathbb{R}^d\to{\Omega_\bullet}}}
	& C^{\infty}_E (\Omega)
	\arrow{ld}{\operatorname*{extend}_{\Omega\to {\Omega_\bullet}}}
	\\  
	C^{\infty}_E ({\Omega_\bullet})
 \end{tikzcd}\qquad
 \begin{tikzcd}[row sep=large, column sep=huge]
	W^{1,p}_E (\mathbb{R}^d)
 	\arrow{r}{\operatorname*{restrict}_{\mathbb{R}^d\to\Omega}} 
 	\arrow{d}[swap]{\operatorname*{restrict}_{\mathbb{R}^d\to{\Omega_\bullet}}}
 	& W^{1,p}_E (\Omega)
 	\arrow{ld}{\operatorname*{extend}_{\Omega\to {\Omega_\bullet}}}
 	\\  
 	W^{1,p}_E ({\Omega_\bullet})
 \end{tikzcd}
 \end{equation*}
\end{enumerate}
\end{proof}
\begin{rem} \label{r-arbitrar}
 \begin{enumerate}
  \item The reader should notice that no assumptions on $E$ beside compactness
	  are necessary.
 % \item \label{r-arbitrar:ii} Since $\Xi \cap E = \emptyset$ the union
%	$\partial \Omega_\bullet = \Xi \cup ( E \cap \partial \Omega_\bullet)$
%	is a disjoint one. So, as $\Xi$ is relatively open in $\partial \Omega$
%	the set $E \cap \partial \Omega_\bullet$ is closed.\footnote{Zwecks
%	sp\"aterem Zitat eingef\"ugter Punkt}
  \item Observe that, after having extended the functions, being defined on
	$\Omega$, to $\Omega_\bullet$, the 'Dirichlet crack' $\Sigma$ in 
	Figure~\ref{fig-Meissel} has vanished, and one ends up with the whole
	cube. Here the problem of extending Sobolev functions is almost
	trivial. We suppose that this is the generic case -- at least for
	applied problems.
 \end{enumerate} 
\end{rem}
The above considerations suggest the following procedure: extend the functions
from $W^{1,p}_E(\Omega)$ first to $\Omega_\bullet$, and afterwards to the whole
of $\R^d$. The next lemma will show that this approach is universal.
\begin{lemma} \label{l-exten0}
 Every linear, continuous extension operator $\mathfrak F : W^{1,p}_E(\Omega)
 \to W^{1,p}_E(\R^d)$ factorizes in the following manner: there is a linear,
 continuous extension operator ${\mathfrak F}_\bullet :
 W^{1,p}_E(\Omega_\bullet) \to W^{1,p}_E(\R^d)$, such that $\mathfrak F =
 {\mathfrak F}_\bullet \mathrm{Ext}(\Omega, \Omega_\bullet)$.
\end{lemma}
\begin{proof}
Let $\mathfrak S$ be the restriction operator from $W^{1,p}_E(\Omega_\bullet)$
to $W^{1,p}_E(\Omega)$. Then we define, for every $f \in W^{1,p}_E(\Omega_\bullet)$,
$\mathfrak F_\bullet f:=\mathfrak F \mathfrak Sf$. We obtain
$\mathfrak F_\bullet \mathrm{Ext}(\Omega, \Omega_\bullet)=\mathfrak F \mathfrak S
\mathrm{Ext}(\Omega, \Omega_\bullet) =\mathfrak F$.
This shows that the factorization holds algebraically. But one also has
 \begin{align*}
   \|{\mathfrak F}_\bullet \mathrm{Ext}(\Omega, \Omega_\bullet) f
	\|_{W^{1,p}_E(\R^d)} &= \| \mathfrak F f \|_{W^{1,p}_E(\R^d)} \le
	\| \mathfrak F \|_{\mathcal L(W^{1,p}_E(\Omega);W^{1,p}_E(\R^d))}
	\|f\|_{W^{1,p}_E(\Omega)} \\
   &= \|\mathfrak F \|_{\mathcal L(W^{1,p}_E(\Omega);W^{1,p}_E(\R^d))}
	\| \mathrm{Ext}(\Omega, \Omega_\bullet) f
	\|_{W^{1,p}_E(\Omega_\bullet)}.
  \qedhere
 \end{align*}
\end{proof}
Having extended the functions already to $\Omega_\bullet$, one may proceed as
follows: $E$ was compact, thus $E_\bullet := E \cap \partial \Omega_\bullet$ is
closed in $\partial \Omega_\bullet$.
%closed, cf.\@ Remark~\ref{r-arbitrar}\ref{r-arbitrar:ii} 
So one can now consider the space $W^{1,p}_{E_\bullet}(\Omega_\bullet)$ 
and has then the task to establish an extension operator for this space -- while
afterwards taking into account that the original functions were $0$ also on
the set $E \cap \Omega_\bullet$.
%\end{rem}
\begin{definition} \label{d-sobext}
 Let $\Lambda \subset \R^d$ be a bounded domain and suppose $p \in [1,\infty[$.
 Then we say that $\Lambda$ is a $W^{1,p}$-extension domain, if there exists a
 linear, continuous extension operator $\mathfrak F_p : W^{1,p}(\Lambda) \to
 W^{1,p}(\R^d)$. We call $\Lambda$ a \emph{Sobolev extension domain}, if it is
 a $W^{1,1}$-extension domain and the restrictions of $\mathfrak F_1$ give
 continuous extension operators $\mathfrak F_p : W^{1,p}(\Lambda) \to
 W^{1,p}(\R^d)$ for any $p \in {]1,\infty[}$.
\end{definition}

We now come to the second aim of this subsection, that is to show that a local
extension property of $\Lambda$ around points of
$\overline{\partial \Lambda \setminus F}$ already gives an extension operator
for the space $W^{1,p}_F(\Lambda)$. Here one should think of $\Lambda$ as being 
either $\Omega$ or $\Omega_\bullet$.

\begin{theorem} \label{l-extension}
 Fix $ p \in [1,\infty[$. Let $\Lambda$ be a bounded domain and let $F$ be a
 closed part of its boundary. Assume that, for every $\mathrm x \in
 \overline{\partial \Lambda \setminus F}$, there is an open neighbourhood
 $U_\mathrm x$ of $ \mathrm x$ such that $\Lambda \cap U_\mathrm x$ is a
 $W^{1,p}$-extension domain. Then there is a continuous extension operator
 $\mathfrak F_p : W^{1,p}_F(\Lambda) \to W^{1,p}(\R^d)$.
\end{theorem}
\begin{proof}
 For every $\mathrm x \in \overline{\partial \Lambda \setminus F}$, let
 $U_\mathrm x$ be the open neighbourhood of $\mathrm x$ from the assumption.
 Let $U_{\mathrm x_1}, \ldots, U_{\mathrm x_n}$ be a finite subcovering of 
 $\overline{\partial \Lambda \setminus F}$. Since the compact set
 $\overline{\partial \Lambda \setminus F}$ is contained in the open set
 $\bigcup_j U_{\mathrm x_j}$, there is an $\epsilon > 0$, such that the sets
 $U_{\mathrm x_1},\ldots,U_{\mathrm x_n}$, together with the open set $U :=
 \{\mathrm y \in \Omega : \mathrm{dist}(\mathrm y, \overline{\partial \Lambda
 \setminus F}) > \epsilon \}$, form an open covering of $\overline \Lambda$.
 Hence, on $\overline \Lambda$ there is a $C^\infty$-partition of unity $\eta,
 \eta_1, \ldots, \eta_n$, with the properties $\mathrm{supp}(\eta) \subset U$,
 $\mathrm{supp}(\eta_j) \subset U_{\mathrm x_j}$.

 Assume $\psi \in C^\infty_F(\Lambda)$. Then $\eta \psi \in C^\infty_0(\Lambda)
 \subset W^{1,p}_0(\Lambda)$. If one extends this function by $0$ outside of
 $\Lambda$, then one obtains a function $\varphi \in
 C^\infty_{\partial \Lambda}(\R^d) \subset C^\infty_F(\R^d) \subset
 W^{1,p}_F(\R^d)$ with the property $\|\varphi\|_{W^{1,p}(\R^d)} =
 \|\eta \psi\|_{W^{1,p}(\Lambda)}$.

 Now, for every fixed $j \in \{1, \ldots, n\}$, consider the function $\psi_j
 := \eta_j \psi \in W^{1,p}(\Lambda \cap U_{\mathrm x_j})$. Since $\Lambda \cap
 U_{\mathrm x_j}$ is a $W^{1,p}$-extension domain by supposition, there is an
 extension of $\psi_j$ to a $W^{1,p}(\R^d)$-function $\varphi_j$ together with
 an estimate $\|\varphi_j\|_{W^{1,p}(\R^d)} \le c
 \|\psi_j\|_{W^{1,p}( \Lambda \cap U_{\mathrm x_j})}$, where $c$ is independent
 from $\psi$. Clearly, one has a priori no control on the behaviour of
 $\varphi_j$ on the set $\Lambda \setminus U_{\mathrm x_j}$. In particular
 $\varphi_j$ may there be nonzero and, hence, cannot be expected to coincide
 with $\eta_j \psi$ on the whole of $\Lambda$. In order to correct
 this, let $\zeta_j$ be a $C^\infty_0(\R^d)$-function which is identically $1$
 on $\mathrm{supp}(\eta_j)$ and has its support in $U_{\mathrm x_j}$. Then
 $\eta_j \psi$ equals $\zeta_j \varphi_j $ on all of $\Lambda$. Consequently,
 $\zeta_j \varphi_j$ really is an extension of $\eta_j \psi$ to the whole of
 $\R^d$ which, additionally, satisfies the estimate
 \[ \| \zeta_j \varphi_j \|_{W^{1,p}(\R^d)} \le c \| \varphi_j
	\|_{W^{1,p}(\R^d)}\le c \| \eta_j \psi
	\|_{W^{1,p}(\Lambda \cap U_{\mathrm x_j})} \le c
	\|\psi\|_{W^{1,p}(\Lambda)},
 \]
 where $c$ is independent from $\psi$. Thus, defining $\mathfrak F_p(\psi) =
 \varphi + \sum_j \zeta_j \varphi_j$ one gets a linear, continuous extension
 operator $\mathfrak F_p$ from $C^\infty_F(\Lambda)$ into $W^{1,p}(\R^d)$. By
 density, $\mathfrak F_p$ extends uniquely to a linear, continuous operator
 \begin{equation} \label{e-extendi}
  \mathfrak F_p:W^{1,p}_F(\Lambda) \to W^{1,p}(\R^d).
\end{equation} 
\end{proof}

\begin{rem} \label{r-unif}
 \begin{enumerate} 
\item
         Observe that the set $F=\partial \Omega_\bullet \cap E$ is \emph{not}
         necessarily again a $(d-1)$-set; possibly one
         even has $\mathcal H_{d-1}(F)=0$. (Take Figure 2 and suppose that this 
         time only $\Sigma$ forms the whole Dirichlet part of the boundary.)
         The reader should carefully notice, 
         that this does not affect the considerations in Theorem \ref{l-extension}.
  \item Of course, one gets uniformity with respect to $p$ from any subinterval
	of $[1,\infty[$ if one invests uniformity concerning the
	extension property for the local domains $\Lambda \cap U_\mathrm x$.
  \item \label{r-unif:ii} If one aims at an extension operator $\mathfrak E :
	W^{1,p}_D(\Omega) \to W^{1,p}_D(\R^d)$, one is free to modify the
	domain $\Omega$ to $\Omega_\bullet$ -- or not. In most cases the local
	geometry improves (concerning Sobolev extension), but we are unable to
	show that this is always the case -- irrespective of
	Lemma~\ref{l-exten0}. On the other hand, we have no examples where the
	situation becomes worse.
 \end{enumerate}
\end{rem}

Theorem~\ref{l-extension} yields a Sobolev extension operator from
$W^{1,p}_F(\Lambda)$ to $W^{1, p}(\R^d)$. However, our aim is to show that it
does not destroy the boundary behavior, which means that it even maps to
$W^{1,p}_F(\R^d)$. We will now turn to this question.

\begin{lemma} \label{l-trace}
 Suppose that $\Lambda \subset \R^d$ is a domain and $F \subseteq \partial \Lambda$
is a (closed) $(d-1)$-set. Moreover, assume that for
 $\mathcal H_{d-1}$-almost all points $\mathrm y \in F$, balls around
 $\mathrm y$ in $\Lambda$ have \emph{asymptotically nonvanishing relative
 volume}, i.e.
 \begin{equation} \label{e-asymp}
  \liminf_{r \mapsto 0} \frac {|B(\mathrm y;r) \cap \Lambda)|}{r^d} > 0.
 \end{equation}
 Let $\psi \in C^\infty_F(\Lambda)$ and $p \in {]1,\infty[}$. If there is an
 extension $\widehat \psi \in W^{1,p}(\R^d)$ of $\psi$, then $\widehat \psi \in
 W^{1,p}_F(\R^d)$.
\end{lemma}
\begin{proof}
 One first proves the property $\widehat \psi \in \mathcal W^{1,p}_F(\R^d)$.
 In fact, since $\mathrm{supp}(\psi)$ has a positive distance to $F$, one
 clearly has $\lim_{r \to 0} \frac{1}{|\Lambda \cap B(\mathrm y,r)|}
 \int_{\Lambda \cap B(\mathrm y,r)} \psi(\mathrm x) \; d \mathrm x = 0$ for all
 $\mathrm y \in F$. From this, one deduces
 \begin{equation} \label{e-outtrace}
  \lim_{r \to 0} \frac{1}{|B(\mathrm y,r)|} \int_{B(\mathrm y,r)} \widehat \psi
	(\mathrm x) \; d \mathrm x = 0 \quad \text{ for } \quad
	\mathcal H_{d-1}\text{-almost all }  \mathrm y \in F.
\end{equation}
 The proof of this runs along the same lines as the proof of
 \cite[Ch.~VIII Prop.~2]{jons}; with two differences:

 Firstly, one has to take the measure $\mu$ here as $\mathcal H_{d-1}|_F$
 instead of $\mathcal H_{d-1}|_{\partial \Lambda}$ there. In order to do so, one
 has to show that this measure has the required functional analytic quality --
 and this is the case.

 Secondly, one observes that the $\liminf$ in \eqref{e-asymp} does in fact
 not need not to have a uniform lower bound for ($\mathcal H_{d-1}$-almost) all
 $\mathrm y \in F$, the above condition suffices.

 But \eqref{e-outtrace} implies $\widehat \psi \in \mathcal W^{1,p}_F(\R^d)$,
 recall Remark~\ref{r-wichtig}. Having this at hand, one applies 
Theorem~\ref{t-coincid}.
\end{proof}

\begin{corollary} \label{c-extendtrace}
 Assume that there is a linear, continuous extension operator $\mathfrak E:W^{1,p}_D(\Omega) \to W^{1,p}(\R^d)$.
Then, under the assumptions of Lemma \ref{l-trace}, $\mathfrak E$ maps into the space $W_D^{1,p}(\R^d)$.
\end{corollary}

In the case where an extension operator from $W^{1,p}_F(\Omega) $ into $W^{1,p}(\R^d) $
operates \emph{uniformly in $p$}, things become much simpler -- as the following
result shows:
\begin{lemma} \label{l-moritz}
Let $\Lambda$ be a bounded domain, and  let $F \subset \partial \Lambda$ be 
a (closed) $(d-1)$-set. If there is continuous extension operator $\mathfrak E:
W^{1,1}_F(\Lambda) \to W^{1,1}_F(\R^d) $ which acts as a continuous operator
 $\mathfrak E:W^{1,p}_F(\Lambda) \to W^{1,p}(\R^d)$ for all $p \in ]1,d+\epsilon]$
($d$ being the space dimension and $\epsilon >0$). Then, for every $p \in ]1,d+\epsilon]$,
$\mathfrak E$ maps the space $W^{1,p}_F(\Lambda)$ even into $W_F^{1,p}(\R^d)$.
\end{lemma}
\begin{proof}
Fix $p \in ]1,d+\epsilon]$ and assume first $\psi \in C^\infty_F(\Lambda)$. Then 
the extension $\mathfrak E\psi$ does not belong only to $W^{1,p}(\R^d)$ but even
to $W^{1,d+\epsilon}(\R^d)$. Hence, $\mathfrak E\psi$ has a representative which
is H\"older continuous. Moreover, it is clear that this representative is identical
$0$ on $F$. This leads to the property $\mathfrak E \psi \in \mathcal W^{1,p}_F(\R^d)$,
 according to the $(d-1)$-property of $F$, cf. Remark~\ref{r-wichtig}. But, thanks
 to Theorem \ref{t-coincid}, the spaces $\mathcal W^{1,p}_F(\R^d)$
and $W^{1,p}_F(\R^d)$ coincide -- what implies $\mathfrak E \psi \in W^{1,p}_F(\R^d)$.
This proves the assertion for all elements from the dense subspace 
$C^\infty_F(\Lambda)$; what implies the claim by the continuity of $\mathfrak E$.
\end{proof}
\subsection{Geometric conditions}
In this subsection we will present geometric conditions on the boundary part
$\overline {\partial \Lambda \setminus F}$, such that the local sets $\Lambda 
\cap U_{\mathrm x_j}$ really admit the Sobolev extension property required in
Theorem~\ref{l-extension}. A first condition, completely sufficient for the 
treatment of most real world problems, is the following:
\begin{assu} \label{a-extend}
Let $\Omega$ and $\Gamma$ be as in Assumption~\ref{a-1} and let $\Lambda$ be
either $\Omega$ or $\Omega_\bullet$,
cf.\@ Remark~\ref{r-unif}~\ref{r-unif:ii}. For every $\mathrm x \in
\overline \Gamma$ there is an open neighbourhood $U_\mathrm x$ of $\mathrm x$
 and a bi-Lipschitz mapping $\phi_\mathrm x$ from $U_\mathrm x$ onto a cube,
 such that $\phi_\mathrm x(\Lambda \cap U_\mathrm x)$ is the (lower) half cube
 and $\partial \Lambda \cap U_\mathrm x$ is mapped onto the top surface of the
 half cube.
\end{assu}
A proof for the fact that this condition really leads to the required extension
operator is given in \cite{TomJo} for the case $p=2$. It carries over, however,
to $p \in [1,\infty[$ -- word by word.

Another relevant condition that assures the existence of a Sobolev extension
operator is that of Jones \cite{jones}. In order to formulate this we need the
following definition.
\begin{definition} \label{d-Jones}
 Let $\Upsilon \subset \R^d$ be a domain and $\varepsilon, \delta >0$. Assume
 that any two points $\mathrm x, \mathrm y \in \Upsilon$, with distance not
 larger than $\delta$, can be connected within $\Upsilon$ by a rectifiable arc
 $\gamma$ with lenght $l(\gamma)$, such that the following two conditions are
 satisfied for all points $\mathrm z$ from the curve $\gamma$:
 \[ l(\gamma) \le \frac{1}{\varepsilon } \|\mathrm x - \mathrm y\|, \quad
	\text{and} \quad
	\frac{\|\mathrm x - \mathrm z\| \|\mathrm y -\mathrm z\|}
		{\|\mathrm x -\mathrm y\|}
	\le \frac{1}{\varepsilon} \mathrm{dist}(\mathrm z, \Upsilon^c).
 \]
 Then $\Upsilon$ is called $(\varepsilon, \delta)$-domain in the spirit of
 Jones.
\end{definition}

\begin{proposition} \label{p-Jones}
 If $\Upsilon$ is an $(\varepsilon,\delta)$-domain, then it is a Sobolev
 extension domain.
\end{proposition}
\begin{rem} \label{r-uniform}
 This famous result is due to Jones \cite{jones}. \emph{Bounded}
 $(\varepsilon, \delta)$-domains are known  to be \emph{uniform domains}, see
 \cite[Ch.~4.2]{Vai}, compare also \cite{jones}, \cite{martio},
 \cite{martiosarv}, \cite{martin} for further information.

 Although the uniformness property is not necessary for a domain to be a
 Sobolev extension domain (see \cite{Yang}) it seems presently to be the
 broadest class of domains for which this extension property holds -- at least
 if one aims at all $p \in {]1\infty[}$. E.g. it contains Koch's
 snowflake, cf. \cite{jones}
\end{rem}
In view of these considerations, we can formulate the following criteria for
the existence of the required extension operator. cf.\@ Theorems~\ref{l-extension}
and Lemma \ref{l-moritz}.

\begin{theorem} \label{t-domainjones}
 Let $\Omega$, $\Gamma$ and $D$ be as in Assumption \ref{a-1} and let $\Lambda$
 be either $\Omega$ or $\Omega_\bullet$,
 cf.\@ Remark~\ref{r-unif}~\ref{r-unif:ii}. Suppose that Assumption~\ref{a-extend} is
 fullfilled or suppose that for every $\mathrm x \in \overline \Gamma$ there is an open,
 bounded neighbourhood $U_\mathrm x$ of $\mathrm x$, such that $U_\mathrm x
 \cap \Lambda$ is an $(\varepsilon,\delta)$-domain. Then
% under the assumptions of Lemma \ref{l-trace}, 
there exists a continuous, linear extension operator $\mathfrak E : 
W^{1,p}_D(\Omega) \to W_D^{1,p}(\R^d)$.
%\footnote{Ich habe hier am Ende mal das $\Lambda$ in dem
% $W^{1,p}$-Raum durch $\Omega$ ersetzt, denn das ist es ja, was wir eigentlich
% erreichen wollen.}
\end{theorem}
\begin{proof}
Both geometric configurations admit a continuous extension operator 
$W^{1,p}_D(\Omega) \to W^{1,p}(\R^d)$, according \@ Theorem~\ref{l-extension} -- 
which is even uniform in $ p \in ]1,\infty[$. Thus, Lemma \ref{l-moritz} applies.
\end{proof}
\begin{rem} \label{r-vorher}
Lemma \ref{l-moritz}, applied to the special case of Jones' extension operator,
 provides an alternative proof of \cite[Theorem~1.3]{mitrea}
in case of first order Sobolev spaces, if $D$ is a $(d-1)$-set. 
In \cite{mitrea} this is achieved, even for arbitrary compact boundary parts
$D$, by a deep analysis of the support properties 
of the functions obtained by Jones' extension operator.
\end{rem}

\section{Elliptic and parabolic regularity} \label{sec4}

\noindent
In this section we prove that the interpolation property of the spaces
$W^{1,p}_D(\Omega)$ -- in conjuction with a famous result of Sneiberg
\cite{snei} -- already leads to substantial regularity results within this
scale of spaces.

\subsection{Isomorphism property for elliptic operators}
In this subsection we will prove the announced elliptic regularity
theorem, which we consider as the second essential result of this work.
Let us emphasize that spaces like $W^{-1,p}_D$ are adequate for the treatment of 
elliptic/parabolic equations, if the right hand side possibly contains distributional
objects like surface densities. In electrostatics, for instance, a charge
density on an interface causes a jump in the normal component of the
dielectric displacement, see for instance \cite[Chapter~1]{Tam}.

Let us first recall the definition of a scale of Banach spaces (see
\cite[Ch.1]{krein}, compare also \cite[Ch.~1.19.4]{triebel}).

\begin{definition} \label{d-scale}
 Consider a closed interval $I \subset [0,\infty[$ and a family of complex
 Banach spaces $\{ X_\tau \}_{\tau \in I}$. One calls this family a (complex)
 scale (of Banach spaces), if
 \begin{enumerate} 
 \item $X_\beta$ embeds continuously and densely in $X_\alpha$, whenever
	$\beta > \alpha$.
 \item For every triple $\alpha, \beta, \gamma \in I$ satisfying $\alpha <
	\beta < \gamma$ there is a positive constant $c(\alpha,\beta,\gamma)$
	such that for all $\psi \in X_\gamma$ the following interpolation
	inequality holds
	\begin{equation} \label{e-interineq}
	 \| \psi \|_{X_\beta} \le c(\alpha,\beta,\gamma)
		\|\psi\|^\frac{\gamma-\beta}{\gamma-\alpha}_{X_\alpha}
		\|\psi\|_{X_\gamma}^\frac{\beta-\alpha}{\gamma-\alpha}.
	\end{equation}
\end{enumerate}
\end{definition}

\noindent
We associate to the families $\{ W^{1,p}_D(\Omega) \}_{p \in {]1,\infty[}}$
and $\{ W^{-1,p}_D(\Omega) \}_{p\in {]1,\infty[}}$ Banach scales in the
following manner

\begin{definition} \label{d-scaleW}
 For $\tau \in {]0,1[}$ we define $X_\tau := W^{1, (1-\tau)^{-1}}_D(\Omega)$
 and $Y_\tau := W^{-1, (1-\tau)^{-1}}_D(\Omega)$.
\end{definition}

\begin{lemma} \label{l-scale}
 Let Assumptions~\ref{a-1} and \ref{a-extendgeneral} be satisfied. Then, for
 all $\tau_1, \tau_2 \in {]0,1[}$ with $\tau_1 < \tau_2$, the families
 $\{ X_\tau \}_{\tau \in [\tau_1,\tau_2]}$ and
 $\{ Y_\tau \}_{\tau \in [\tau_1,\tau_2]}$ form Banach scales.
\end{lemma}

\begin{proof}
 We show more, namely: for every $\alpha, \beta, \gamma \in {]0,1[}$ with
 $\alpha < \beta < \gamma$ one has
 \begin{equation} \label{e-intscle}
  X_\beta = [ X_\alpha, X_\gamma ]_\frac{\beta-\alpha}{\gamma-\alpha}
	\text{ and } Y_\beta = [ Y_\alpha, Y_\gamma
	]_\frac{\beta-\alpha}{\gamma-\alpha}.
 \end{equation}
 Putting $\theta := \frac{\beta-\alpha}{\gamma-\alpha}$, one has $1 - \beta =
 (1 - \alpha)(1 - \theta) + (1-\gamma) \theta$. Thus the equalities in
 \eqref{e-intscle} follow from Theorem~\ref{t-1interpolgen} and
 Corollary~\ref{c-remaintrue-}. The inequality \eqref{e-interineq} is then the
 interpolation inequality for complex interpolation.
\end{proof}
%%%%%%%%%%%%%%%%%%%%%%%%%%%%%%%%%%%%%%%%%%%%%%%%%%%%%%%%%%%%%%%%%%%%%%%%%%%%%%

Throughout the rest of the paper we assume that the following is satisfied:

\begin{assu} \label{a-coeff}
 Let $\mu$ be a matrix valued, bounded, measurable, elliptic function on
 $\Omega$. The latter condition means that $\mathrm{Re} (
 \mu(\mathrm x) \xi\cdot \overline \xi ) \ge \mu_\bullet |\xi|^2$ for some positive
 constant $\mu_\bullet$, all $\xi \in \C^d$ and almost all $\mathrm x \in
 \Omega$.
\end{assu}

\begin{definition} \label{d-opera}
 For every $p \in {]1,\infty[}$, we define the operator $-\nabla \cdot \mu
 \nabla : W^{1,p}_D(\Omega) \to W^{-1,p}_D(\Omega) $ by
 \begin{equation}\label{e-0815}
  \langle - \nabla \cdot \mu \nabla v, w \rangle := \int_\Omega \mu \nabla v
	\cdot \nabla \overline w \; d \mathrm{x}, \quad v \in
	W^{1,p}_D(\Omega), \ w \in W^{1,p'}_D(\Omega),
 \end{equation}
 where $\langle \cdot, \cdot \rangle$ denotes the antidual pairing between 
 $\bigl (W^{1,p'}_D(\Omega)\bigr )' = W^{-1,p}_D(\Omega)$ and
 $W^{1,p'}_D(\Omega)$.
\end{definition}
%%%%%%%%%%%%%%%%%%%%%%%%%%%%%%%%%%%%%%%%%%%%%%%%%%%%%%%%%%%%%%%%%%%%%%%%%%%%%
%
Let us briefly recall the (well known) connection between the operator
$- \nabla \cdot \mu \nabla:W^{1,p}_D(\Omega) \to W^{-1,p}_D(\Omega)$ and mixed 
boundary value problems. For this,
consider the mixed boundary value problem
\begin{align}
  -\nabla \cdot (\mu \nabla u ) &= f_\Omega \in L^p(\Omega)
	\label{e-modelset04} \\
  u|_D  &= 0 \label{e-modelset03} \\
  \nu \cdot \mu \nabla u  &= f_\Gamma \in L^p(\Gamma), \label{e-modelset2}
\end{align}
where $\nu$ denotes the outward unit normal of the domain. Here
\eqref{e-modelset04} is to be understood in the sense of distributions on
$\Omega$ and \eqref{e-modelset2} is to be understood in a generalized sense.
If one defines $f \in W^{-1,p}_D(\Omega)$ by $\langle f, v \rangle :=
\int_\Omega f_\Omega \overline v \; d \mathrm x + \int_\Gamma f_\Gamma
\overline v \; d \mathcal H_{d-1}$, then, under reasonable assumptions on
$\Omega, \Gamma$ and $D$ an adequate functional analytic formulation of the
problem \eqref{e-modelset04}--\eqref{e-modelset2} is the operator equation
$-\nabla \cdot \mu \nabla u = f$, see \cite[Ch.~1.2]{cia}, \cite[Ch~II.2]{ggz})
or \cite{daners} for details; compare also \cite{hoemb} and \cite{elstmeyrreh}
for a different approach.

When restricting the range space of the operator $- \nabla \cdot \mu \nabla $
to $L^2(\Omega)$, one obtains an operator for which the elements $\psi$ of its
domain satisfy the conditions $\psi|_D = 0$ in the sense of traces and $\nu
\cdot \mu \nabla \psi = 0$ on $\Gamma$ in a generalized sense.\\
It follows the second main result of this work:
\begin{theorem} \label{t-groegerrepr}
 Let Assumptions~\ref{a-1}, \ref{a-extendgeneral} and \ref{a-coeff} be
 satisfied. Then there is an open interval $I$ containing $2$, such that the
 operator
 \begin{equation} \label{e-topiso}
  - \nabla \cdot \mu \nabla + 1 : W^{1,p}_D(\Omega) \to W^{-1,p}_D(\Omega)
 \end{equation}
 is a topological isomorphism for all $p \in I$.
\end{theorem}

\begin{proof}
 We know from Lemma~\ref{l-scale} that the families $\{X_\tau \}_{\tau \in
 [\alpha, \beta]}$ and $\{Y_\tau \}_{\tau \in [\alpha, \beta]}$ with $\alpha,
 \beta \in {]0,1[}$ form Banach scales. The mapping in \eqref{e-topiso} is
 continuous for all $p$, due to the boundedness of the coefficient function
 $\mu$, what is to be interpreted as the continuity of
 \begin{equation} \label{e-scael}
  - \nabla \cdot \mu \nabla + 1 : X_\tau \to Y_\tau
 \end{equation}
 for all $\tau \in {]0,1[}$. Lastly, the quadratic form $W^{1,2}_D(\Omega) \ni
 \psi \mapsto \int_\Omega (\mu \nabla \psi \cdot \nabla \overline \psi +
 |\psi|^2 ) \; d \mathrm x$ is coercive. Hence the Lax-Milgram lemma gives the
 continuity of the inverse of \eqref{e-topiso} in the case of $p = 2$. In the
 scale terminology, this is nothing else but the continuity of
 $(-\nabla \cdot \mu \nabla + 1)^{-1} : Y_\tau \to X_\tau$ in case of $\tau =
 \frac12$. A deep theorem of Sneiberg (\cite{snei}, see also
 \cite[Lemma~4.16]{auschmem} or \cite{vigna}) says that the set of parameters
 $\tau$ for which \eqref{e-scael} is a topological isomorphism, is open. Since
 $\frac12$ is contained in this set, it cannot be empty.
\end{proof}

\begin{rem} \label{r-notlarg}
 \begin{enumerate} 
 \item Again interpolation shows that the values $p$, for which
	\eqref{e-topiso} is a topological isomorphism, form an interval $I$.
	Due to the Sneiberg result, this interval is an open one.
 \item If $\mu$ takes real, symmetric matrices as values, then
	\begin{equation} \label{e-op}
	  - \nabla \cdot \mu \nabla + 1 : W^{1,p'}_D(\Omega) \to
		W^{-1,p'}_D(\Omega)
	\end{equation} 
	is the adjoint to
	\begin{equation} \label{e-op'}
	  - \nabla \cdot \mu \nabla + 1 : W^{1,p}_D(\Omega) \to
		W^{-1,p}_D(\Omega)
	\end{equation}
	with respect to the sesquilinear pairing. Hence, \eqref{e-op} is a
	topological isomorphism iff \eqref{e-op'} is. Thus, the interval $I$
	must be of the form $I = {]\frac{q}{q-1},q[}$ for some $q > 2$.
 \item It is well-known that the interval $I$ depends on the domain $\Omega$
	(see \cite{jer/ke}, \cite{dauge}) as well as on $\mu$ (see \cite{mey}
	or \cite{mazelresc}), and on $D$ (see \cite{mitr}). The most important
	point is that the length of $I$ may be arbitrarily small, see
	\cite[Ch.~4]{e/r/s} for a striking example. Even in smooth situations
	it cannot be expected that $4$ belongs to $I$, as the pioneering
	counterexample in \cite{shamir} shows.
 \item \label{r-notlarg:iv} If $\mathcal M$ is a set of coefficient functions
	$\mu$ with a common $L^\infty$ bound and a common ellipticity constant,
	then one can find a common open interval $I_\mathcal M$ around $2$,
	such that \eqref{e-topiso} is a topological isomorphism for all $\mu
	\in \mathcal M$ and all $p \in I_\mathcal M$. Finally, one has
	\[ \sup_{ p \in I_\mathcal M} \sup_{\mu \in \mathcal M} \bigl\|
		(-\nabla \cdot \mu \nabla + 1)^{-1}
		\bigr\|_{\mathcal L(W^{-1,p}_D; W^{1,p}_D)} < \infty.
	\]
	The proof of this is completely analogous to \cite[Thm.~1]{groe/reh}.
\item It is interesting to observe that in the case of two space dimensions
	Theorem~\ref{t-groegerrepr} immediately implies the H\"older
	continuity of the solution as long as the right hand side belongs to a
	space $W^{-1,p}_D(\Omega)$ with $p > 2$. The question arises whether
	this remains true in higher dimensions, for $p$ exceeding the
	corresponding space dimension -- despite the fact that the gradient of
	the solution does only admit integrability a bit more than $2$ in
	general. We will prove -- by entirely different methods -- in a
	forthcoming paper \cite{Jo} that this is indeed the case.
\end{enumerate}
\end{rem}

\begin{corollary} \label{c-spec}
 Let Assumptions~\ref{a-1}, \ref{a-extendgeneral} and \ref{a-coeff} be
 satisfied and let $I$ denote the interval guaranteed by
 Theorem~\ref{t-groegerrepr}. Then the following holds
 \begin{enumerate} 
 \item The operator
	\begin{equation} \label{e-topisolam}
	  - \nabla \cdot \mu \nabla + \lambda : W^{1,p}_D(\Omega) \to
		W^{-1,p}_D(\Omega)
	\end{equation}
	is a topological isomorphism for all $p \in I \cap [2,\infty[$, if
	$- \lambda \in \C$ is not an eigenvalue of $-\nabla \cdot \mu \nabla$.
 \item If $0$ is the only constant function in the space $W^{1,2}_D(\Omega)$,
%, in particular only one point in $D$ possess a bi-Lipschitzian chart around, 
then $0$ is not an eigenvalue of $-\nabla \cdot \mu \nabla$.
%	\footnote{Sollten wir das nicht lieber so formulieren: Wenn in
%	$W^{1,p}_D$ Null die einzige Konstante ist, dann gilt\dots Irgendwie
%	hoffe ich ja immernoch, dass daf\"ur die $d-1$-set-Eigenschaft von $D$
%	ausreicht.}
\end{enumerate}
\end{corollary}

\begin{proof}
 \begin{enumerate}
 \item According to Remark~\ref{r-remain}~\ref{r-remain:iii}, the embeddings
	$W^{1,p}_D(\Omega) \hookrightarrow L^p(\Omega) \hookrightarrow
	W^{-1,p}_D(\Omega)$ are compact. Thus \eqref{e-topisolam} can only fail to be an
	isomorphism, if $-\lambda$ is an eigenvalue for $-\nabla \cdot \mu
	\nabla$, according to the Riesz-Schauder theory,
	cf.\@ \cite[Ch.~III.6.8]{kato}.	Observe that an eigenvalue for
	$- \nabla \cdot \mu \nabla$, when considered on $W^{-1,p}_D(\Omega)$
	for $p > 2$ is automatically an	eigenvalue when $-\nabla \cdot \mu
	\nabla$ is considered on $W^{-1,2}_D(\Omega)$.
%Since all eigenvalues are real, the assertion
%	for $p \in I \cap {]1,2[}$ follows by duality.
 \item Assume that this is false, and let $w \in \dom_{W^{-1,p}_D(\Omega)}
	(\nabla \cdot \mu \nabla) \subset W^{1,2}_D(\Omega)$ be the
	corresponding eigenfunction. Then, testing the equation $-\nabla \cdot
	\mu \nabla w = 0$ by $w$, one gets
	\[ 0 = \langle -\nabla \cdot \mu \nabla w, w \rangle \ge c \int_\Omega
		\| \nabla w \|^2 \; \mathrm{d} \mathrm x,
	\]
	thanks to the ellipticity of $\mu$. Hence, $w$ has to be constant on
	$\Omega$, and, consequently, must be $0$ by supposition. 
%But, the $(d-1)$-property of $D$ together with the assumed
%	regularity of the boundary near one point $\mathrm y \in D$, imply that
%	nonzero constant functions cannot belong to $W^{1,2}_D(\Omega)$ and,
%	hence, not to $\dom_{W^{-1,p}_D(\Omega)} (-\nabla \cdot \mu \nabla)
%	\subset W^{1,2}_D(\Omega)$.\footnote{Wenn Umbau in Behauptung, ist
%	nat\"urlich hier Anpassung f\"allig.}
	\qedhere
\end{enumerate}
\end{proof}
\begin{rem} \label{c-relevant}
$W^{1,2}_D(\Omega)$ does not contain any nonzero constant function if $D$ is a
$(d-1)$-set and the boundary around only one point in $D$ possesses a 
bi-Lipschitzian chart around.
\end{rem}
\subsection{Analytic semigroups}

In the sequel we are going to show how to exploit the elliptic regularity
result for proving resolvent estimates for the operators $-\nabla \cdot \mu
\nabla$, which assure the generator property for an analytic semigroup on
suitable spaces $W^{-1,p}_D(\Omega)$. It is well known that this property
allows to solve parabolic equations like
\[ u' - \nabla \cdot \mu \nabla u = f; \quad u(0) = u_0,
\]
where the right hand side $f$ depends H\"older continuously on time (or even
suitably on the solution $u$ itself), see \cite{luna} or \cite{henry}. Since
we proceed very similar to \cite{groe/reh} we do not point out all details but
refer to that paper.

\begin{theorem} \label{t-analythg}
 Let Assumptions~\ref{a-1}, \ref{a-extend} and \ref{a-coeff} be satisfied.
 Suppose, additionally, that $\overline \Omega \subset \R^d$ is a $d$-set.
 Then the following assertions hold true.
 \begin{enumerate} 
  \item \label{t-analythg:i} There is an open interval $J$ containing $2$, such
	that the operator $\nabla \cdot \mu \nabla$ generates an analytic
	semigroup on $W^{-1,p}_D(\Omega)$, as long as $p \in J$.
  \item \label{t-analythg:ii} If $\mathcal M$ is a set of coefficient functions
	$\mu$ with common $L^\infty$ bound and common ellipticity constant, one
	can find -- in the spirit of \ref{t-analythg:i} -- a common interval
	$J_\mathcal M$ for all these $\mu \in \mathcal M$.
  \item \label{t-analythg:iii} There is an open interval $J_\mathcal M$
	containing $2$ such that for all $p \in J_\mathcal M$ one has resolvent
	estimates like
	\begin{equation} \label{e-resestim}
	  \| \bigl( -\nabla \cdot \mu \nabla + 1 +\lambda \bigr)^{-1}
		\|_{\mathcal L(W^{-1,p}_D(\Omega))} \le \frac{c}{1+|\lambda|},
	\end{equation}
	which are uniform in $\mu \in \mathcal M$, $p \in J_\mathcal M$ and
	$\lambda \in \C_+ := \{\vartheta \in \C : \mathrm{Re}(\vartheta) \ge
	0\}$, i.e.\@ the same constant $c$ may be taken for all these
	parameters.
 \end{enumerate}
\end{theorem}

\begin{proof}
 Assertion~\ref{t-analythg:iii} implies points \ref{t-analythg:i} and
 \ref{t-analythg:ii}, so we concentrate on this.
 Concerning the $p$'s above $2$ one proceeds exactly as in \cite{groe/reh}:
 Assumption~\ref{a-extend} provides an extension operator $\mathfrak E$ which
 acts continuously between the spaces $W^{1,p}_D(\Omega)$ and
 $W^{1,p}_D(\R^d)$, cf.\@ Theorem~\ref{t-domainjones}. This leads, via
 Corollary~\ref{c-spec}, to a (nontrivial) interval $I_0 := [2,p_0]$, such that
 \eqref{e-topisolam} is a topological isomorphism for all $p \in I_0$ and all
 $\lambda \in \C_+$. In a next step, to the operators $-\nabla \cdot \mu
 \nabla$ we will associate operators on the extended domain $\widetilde \Omega
 := \Omega \times {]0,1[}$. The 'extended' boundary part $\widetilde \Gamma$ we
 define as $\widetilde \Gamma := \Gamma \times {]0,1[}$, thus obtaining
 \begin{equation} \label{e-Rand}
   \widetilde D := \partial \widetilde \Omega \setminus \widetilde \Gamma =
	\bigl( \overline{\Omega} \times \{0,1\} \bigr) \cup \bigl( D \times
	{]0,1[} \bigr) = \bigl( \overline{\Omega} \times \{0,1\} \bigr) \cup
	\bigl( D \times [0,1] \bigr).
 \end{equation}
 Since $\overline{\Omega}$ is a $d$-set and $D$ is a $(d-1)$-set by
 supposition, it is clear that $\widetilde D$ is a $d$-set. Moreover, it is not
 hard to see that $\widetilde{\Gamma}$ satisfies (mutatis mutandis) the
 condition in Assumption~\ref{a-extend}.

 The following considerations can be carried out in detail in exactly the
 same way as in \cite{groe/reh}, and we give here only a short summary of the
 main steps. As in \cite{groe/reh}, for every $\lambda \in \C_+$ and $\mu \in
 \mathcal M$, one defines a coefficient function $\widetilde \mu$ on
 $\widetilde \Omega$ in the following manner: Let $\mu^\bullet$ be the
 $L^\infty$ bound for the coefficient function $\mu$ and $\mu_\bullet$ its
 ellipticity constant. Then we introduce the coefficient function for the
 auxiliary divergence operator on $\widetilde \Omega$ by
 \begin{equation} \label{e-coeffi}
   \widetilde \mu_{j,k} (\mathrm x,t) = \begin{cases}
	  (1 - \frac{\mu_\bullet}{2\mu ^\bullet} \sign(\mathrm{Im}(\lambda)) i)
		\mu_{j,k} (\mathrm x), & \text {if } j,k \in \{1, \dots, d\},
		\\
	  0 & \text{if } j = d + 1 \text{ or } k = d + 1, \\
	  \frac{\lambda}{|\lambda|} \bigl( \mu^\bullet - \frac{\mu_\bullet}{2}
		 \sign(\mathrm{Im}(\lambda)) i \bigr), & \text{if } j = k =
		d + 1.
	\end{cases}
 \end{equation}

 One easily observes that all these coefficient functions admit $L^\infty$
 bounds  and ellipticity constants that are uniform in $\lambda$. Thus,
 Remark~\ref{r-notlarg}~\ref{r-notlarg:iv} applies to the operators $- \nabla
 \cdot \widetilde \mu \nabla + 1$. This gives an interval $I_1 := [2, p_1]$
 such that the norms of the operators $( -\nabla \cdot \widetilde \mu \nabla +
 1 )^{-1} : W^{-1,p}_{\widetilde D}(\widetilde \Omega) \to
 W^{1,p}_{\widetilde D}(\widetilde \Omega)$ are bounded, uniformly in $\lambda
 \in \C_+ $ and in $p \in I_1$.

 One associates to the problem $(-\nabla \cdot \mu \nabla + 1 + \lambda )u =
 f$ a problem $(-\nabla \cdot \widetilde \mu \nabla + 1) u_\lambda =
 f_\lambda$ and exploits the (uniform) regularity properties of the operators
 $-\nabla \cdot \widetilde \mu \nabla + 1$ for an estimate
 \begin{equation} \label{e-inverscont}
   \|u\|_{W^{1,p}_D(\Omega)} \le c \|f \|_{W^{-1,p}_D(\Omega)},
 \end{equation}
 where
%\footnote{Index $\lambda$ dran oder nicht noch zu kl\"aren}
 $c$ is independent from $f$ and $\lambda \in \C_+$. We already know 
the isomorphism property
 \[ - \nabla \cdot \mu \nabla + 1+\lambda : W^{1,p}_D(\Omega) \to
	W^{-1,p}_D(\Omega),
 \]
 thus \eqref{e-inverscont} may be expressed as
 \begin{equation} \label{e-opnorm}
   \sup_{\lambda \in \C_+} \bigl\| ( -\nabla \cdot \mu \nabla + 1 + \lambda
	)^{-1} \bigr\|_{\mathcal L(W^{-1,p}_D(\Omega); W^{1,p}_D(\Omega))} <
	\infty
 \end{equation}
 for all $p \in I_0 \cap I_1$.
% Since these considerations repeat
%those in \cite{groe/reh} word by word, we refer to this paper.

 Finally, \eqref{e-opnorm} allows us to deduce the estimate
 \begin{align*}
   & \sup_{\lambda \in \C_+} |\lambda| \bigl\| ( -\nabla \cdot \mu \nabla
	+ 1 + \lambda )^{-1} \bigr\|_{\mathcal L (W^{-1,p}_D(\Omega))} \\
   =\strut& \sup_{\lambda \in \C_+} \bigl\| \lambda (-\nabla \cdot \mu \nabla
	+ 1 + \lambda )^{-1} \bigr\|_{\mathcal L (W^{-1,p}_D(\Omega))} \\
   =\strut& \sup_{\lambda \in \C_+} \bigl\| 1 - ( -\nabla \cdot \mu \nabla
	+ 1 ) ( -\nabla \cdot \mu \nabla + 1 + \lambda )^{-1}
	\bigr\|_{\mathcal L(W^{-1,p}_D(\Omega))} \\
   \le\strut&  1 + \| {-\nabla} \cdot \mu \nabla + 1
	\|_{\mathcal L(W^{1,p}_D(\Omega); W^{-1,p}_D(\Omega))}
	\sup_{\lambda \in \C_+} \bigl\| ( -\nabla \cdot \mu \nabla + 1 +
	\lambda )^{-1}
	\bigr\|_{\mathcal L(W^{-1,p}_D(\Omega); W^{1,p}_D(\Omega))} \\
   <\strut& \infty
 \end{align*}
 for all $p \in I_0 \cap I_1$.

 The case $p<2$ can be treated as follows: first, one gets the following
 resolvent estimate on $W^{1,p}_D(\Omega)$ for $p>2$:
 \begin{align*}
   & \bigl\| ( -\nabla \cdot \mu \nabla + 1 + \lambda )^{-1}
	\bigr\|_{\mathcal L(W^{1,p}_D(\Omega))} \\
   =\strut& \bigl\| ( -\nabla \cdot \mu \nabla  + 1 )^{-1} ( -\nabla \cdot \mu
	\nabla + 1 + \lambda  )^{-1} ( -\nabla \cdot \mu \nabla + 1 )
	\bigr\|_{\mathcal L(W^{1,p}_D(\Omega))} \\
   \le\strut& \bigl\| ( -\nabla \cdot \mu \nabla + 1 )^{-1}
	\bigr\|_{\mathcal L(W^{-1,p}_D(\Omega); W^{1,p}_D(\Omega))}
	\bigl\| ( -\nabla \cdot \mu \nabla + 1 + \lambda )^{-1}
	\bigr\|_{\mathcal L(W^{-1,p}_D(\Omega))} \\
   &\hspace{6cm} \times
	\| {-\nabla \cdot \mu \nabla + 1}
	\|_{\mathcal L(W^{1,p}_D(\Omega);W^{-1,p}_D(\Omega))}.
 \end{align*}
 Since the first and third factor are finite, one can use \eqref{e-resestim}.
 Then, considering the adjoint of $( -\nabla \cdot \mu \nabla + 1 + \lambda
 )^{-1}$, which is nothing else but $( -\nabla \cdot \mu^* \nabla + 1 +
 \overline \lambda )^{-1}$ on $W^{-1,p'}_D(\Omega)$, one obtains the assertion
 for $p<2$.
\end{proof}

\begin{rem} \label{r-agmon}
% \begin{enumerate} 
%  \item 
        One could take the suppositions in Theorem~\ref{t-analythg} even more
	general. What in fact is needed is that also the spaces
	$W^{1,p}_{\widetilde D}(\widetilde \Omega)$ possess extension
	operators. This follows in case of Assumption~\ref{a-extend} in a
	peculiarly simple way since it is self-reproducing when passing to the
	set $\Omega \times {]0,1[}$.
%  \item The proof of the main result in \cite{groe/reh} follows an old idea of
%	Agmon in \cite{agmon}.
% \end{enumerate}
\end{rem}

%%%%%%%%%%%%%%%%%%%%%%%%%%%%%%%%%%%%%%%%%%%%%%%%%%%%%%%%%%%%%%%%%%%%%%%%%%%%%%
%%%%%%%%%%%%%%%%%%%%%%%%%%%%%%%%%%%%%%%%%%%%%%%%%%%%%%%%%%%%%%%%%%%%%%%%%%%%%
%%%%%%%%%%%%%%%%%% SYSTEME %%%%%%%%%%%%%%%%%%%%%%%%%%%%%%%%%%%%%%%%%%%%%%%%%%
%%%%%%%%%%%%%%%%%%%%%%%%%%%%%%%%%%%%%%%%%%%%%%%%%%%%%%%%%%%%%%%%%%%%%%%%%%%%%%
%%%%%%%%%%%%%%%%%%%%%%%%%%%%%%%%%%%%%%%%%%%%%%%%%%%%%%%%%%%%%%%%%%%%%%%%%%%%%

\section{Elliptic regularity for systems}
\label{sec_systems}
In this section we apply the interpolation property of the
$W^{1,p}$-spaces in order to derive $p$-estimates for linear elliptic
operators acting on vector-valued functions. Here, for each component a
different Dirichlet boundary might be prescribed. To be more precise, we
assume the following
\begin{itemize}
 \item[(A1)] $\Omega\subset\R^d$ is a bounded domain and for $1 \le i \le m$
	the sets $D_i \subset \partial\Omega$ are closed $(d-1)$-sets. Let
	$D := \bigcap_{i=1}^m D_i$ and $\Gamma := \partial\Omega \setminus D$.
	It is assumed that $\Omega$ and $\Gamma$ satisfy
	Assumption~\ref{a-extendgeneral}, cf. Subsection 4.2.
\end{itemize}
For $p\in [1,\infty)$ we introduce the space 
\begin{align*}
 \bbW_D^{1,p}(\Omega)= \prod_{i=1}^m W^{1,p}_{D_i}(\Omega)
\end{align*}
and its dual $\bbW^{-1,p'}_D(\Omega)$ for $\frac{1}{p} + \frac{1}{p'}=1$,
Furthermore, we define the operator
$\calL_p : \bbW_D^{1,p}(\Omega) \to L^p(\Omega; \C^m \times \C^{m\times d })$
by $\calL_p(u)=(u,\nabla u)$. Given a complex valued coefficient function
$\bbA \in L^\infty(\Omega; \text{Lin}\,(\C^{m}\times \C^{m\times d}, \C^{m}
\times \C^{m\times d}))$, we investigate differential operators of the type
\begin{align*}
 \calA : \bbW_D^{1,p}(\Omega) \to \bbW_D^{-1,p}(\Omega),\quad \calA =
	\calL_{p'}^* \bbA \calL_p.
\end{align*}
The corresponding weak formulation on $\bbW^{1,2}_D(\Omega)$ reads $\langle
\calA(u), v \rangle = \int_\Omega \bbA \left( \begin{smallmatrix} u \\ \nabla u
\end{smallmatrix} \right) : \overline{ \left( \begin{smallmatrix} v \\ \nabla v
\end{smallmatrix} \right)} \; d \mathrm{x}$ for $u,v \in \bbW^{1,2}_D(\Omega)$,
where 
\[ (b_1,B_1):(b_2,B_2) = \sum_{i=1}^m b_1^i b_2^i + \sum_{j=1}^m \sum_{k=1}^d
	B_1^{jk} B_2^{jk}
\]
for $(b_1,B_1) , (b_2,B_2)\in \C^m\times \C^{m \times d}$. It is assumed that
the operator $\calA$ is elliptic. More precisely, we assume that
\begin{itemize}
\item[(A2)] There is a constant $\kappa > 0$, such that for all $v \in
  \bbW^{1,2}_D(\Omega)$ it holds $\mathrm{Re} \langle \calA v, v\rangle \geq
  \kappa\norm{v}^2_{\bbW^{1,2}(\Omega)}$.
\end{itemize}
\begin{rem}
We recall that in the case of systems of partial differential equations the
positivity property formulated in (A2) in general does not imply that the
coefficient tensor belonging to the principle part of the differential operator
is positive definite. In general, this coefficient tensor only satisfies the
weaker Legendre-Hadamard condition, cf.\@ \cite{var:GH96a}:
% giaquinta/Hildebrandt    
Assume that (A2) is satisfied for $\bbA = \left( \begin{smallmatrix}
  \bbA_{11} & \bbA_{12} \\
  \bbA_{21} & \bbA_{22} \end{smallmatrix}\right)$, where $\bbA_{22} \in
\mathrm{Lin}\,(\C^{m \times d}, \C^{m \times d})$ corresponds to the principal
part of the operator $\calA$. Then there exists a constant $c_\kappa > 0$, such
that  for all $\xi \in \C^m$ and $\eta \in \C^d$ it holds
\begin{align} \label{dk_e_posdef}
  \mathrm{Re} \big( \bbA_{22} \xi\otimes \eta : \overline{\xi \otimes \eta}
	\big) \geq c_\kappa \abs{\xi}^2 \abs{\eta}^2.
\end{align}
%This means that $\bbA$ is positive definite with respect to 
%rank-one matrices, but
%not necessarily with respect to all matrices in $\R^{m\times d}$. 
\end{rem} 
\begin{theorem} \label{dk_thm1}
 Let (A1) and (A2) be satisfied. Then there exists an open interval $J$
 containing $2$, such that for all $q \in J$ the operator $\calA :
 \bbW_D^{1,q}(\Omega) \to \bbW_D^{-1,q}(\Omega)$ is a topological isomorphism.
\end{theorem}
\begin{proof}
 Exactly the same arguments as in the proof of Theorem \ref{t-groegerrepr} can
 be applied.
\end{proof}
If in addition the operator $\calA$  satisfies a certain symmetry relation,
then the interval $J$ can be determined uniformly for classes of coefficient
tensors satisfying uniform upper bounds and ellipticity properties.
\begin{itemize}
\item[(A3)] 
%For a.e.~$x\in \Omega$ and  
%all $B_1,B_2\in \R^m\times \R^{m\times d}$ it holds $\bbA(x) B_1 :
%B_2 = B_1 : \bbA(x) B_2$.
%\\\\
For all $u,v\in \bbW^{1,2}_D(\Omega)$ it holds
 $\langle\calA u,v \rangle =\overline{\langle\calA
  v,u\rangle}$.
\end{itemize}
%\begin{itemize}
%\item[(A3)] For all $u,v\in \bbW^{1,2}_D(\Omega)$ it holds
%  $\langle\calA(u), v\rangle =\langle u, \calA(v)\rangle$. 
%\end{itemize}
\begin{theorem}
\label{dk_thm2}
Let (A1) be satisfied and let $\calM$ be a set of coefficient tensors
fulfilling (A2) and (A3) with a uniform upper $L^\infty$-bound and a common
lower bound for the ellipticity constant $\kappa$ in (A2). Then, there exists
an open interval $J_\calM$ containing $2$, such that for all $p \in J_\calM$
and all $\bbA\in \calM$ the corresponding operator $\calA$ is a topological
isomorphism between $\bbW^{1,p}_D(\Omega)$ and $\bbW_D^{-1,p}(\Omega)$.
Moreover, there exists a constant $c_\calM>0$ such that for all $f\in
\bbW_D^{-1,p}(\Omega)$ we have 
\begin{align}
\label{dk_est_thm1}
\sup 
\left\{\norm{\calA^{-1}(f)}_{\bbW^{1,p}(\Omega)}\, ;\, p\in
  J_\calM,\,\bbA\in \calM
 \right\} \leq c_\calM\norm{f}_{\bbW_D^{-1,p}(\Omega)}.
\end{align} 
\end{theorem}
\begin{rem}
In the case of scalar equations, i.e.\@ $m=1$, the previous theorem is also
valid for operators $\calA$ which do not satisfy (A3), (see
Remark~\ref{r-notlarg}~\ref{r-notlarg:iv}). Similar arguments as in the scalar
case can be applied to the vectorial case without assuming (A3)
%with non-symmetric coefficients,
provided that the coefficient tensor $\bbA_{22}$ satisfies \eqref{dk_e_posdef}
for all $B\in \C^{m\times d}$ and not only for $B=\xi\otimes \eta$. 
%
% is
%positive definite on the complete space $\R^{m\times d}$ (and not only
%with respect to  rank-one tensors).  
In this case, the proof of the uniform bound \eqref{dk_est_thm1} 
 relies on 
certain estimates that are derived using the positivity of the
coefficient-tensors 
 (see \cite{groe/reh}). In the general non-symmetric vector valued case, 
we do not see how the proof can be generalised,
  if only the weaker positivity \eqref{dk_e_posdef} is assumed. 
In the case studied in Theorem \ref{dk_thm2} 
 we  derive estimates for the corresponding
operators directly (and not pointwise for the coefficients) 
 %Our main argument
 %then relies on
 and  use  the fact that for self-adjoint
operators on a Hilbert space $\bbH$ the operator norm is 
given by $\norm{T}_\text{op} = \sup \left\{\abs{\langle T a,a\rangle}\, ;
    \, a\in \bbH,\, \norm{a}\leq 1\right\}$. % Werner, S.235  
%
%bounded by 
%bounds on the spectrum \cite{ }. 
\end{rem}
\begin{proof}
Let $\calP:\bbW^{1,2}_D(\Omega) \rightarrow \bbW^{-1,2}_D(\Omega)$ be
defined by $\calP=\calL^*\calL$. 
 Due to Theorem \ref{dk_thm1} there exist
$q_0^*<2<q_1^*$ such that for all $p\in [q_0^*,q_1^*]$ the operator
$\calP$ is a topological isomorphism between $\bbW^{1,p}_D(\Omega)$
and $\bbW_D^{-1,p}(\Omega)$. This implies that for all $t>0$ and $p\in
[q_0^*,q_1^*]$ the operator $\calQ_t$, given by $\calQ_t =\calP^{-1}(\calP -
t\calA)$ is a bounded linear operator from $\bbW^{1,p}_D(\Omega)$ to $\bbW^{1,p}_D(\Omega)$. In a first step, we will show that there exist $t_0 >
0$ and $q_0,q_1\in [q_0^*,q_1^*]$ with $q_0<2<q_1$, such that
\begin{align}
\label{dk_est_p1}
\sup_{p\in [q_0,q_1]}\norm{\calQ_{t_0}}_{\text{op},p} \leq \iota<1,
%\sup\{\norm{\calQ_{t_0}(v)}_{\bbW^{1,p}(\Omega)}\,;\, 
%p\in [q_0,q_1], v\in \bbW^{1,p}_D(\Omega),
%\norm{v}_{\bbW^{1,p}(\Omega)}=1\} 
%\leq \rho <1.
\end{align}
where $\norm{\calQ_{t_0}}_{\text{op},p}$ denotes the operator norm 
with respect to the space $\bbW^{1,p}_{D}(\Omega)$. 
In the second step, the uniform estimate \eqref{dk_est_thm1}
 is derived from \eqref{dk_est_p1}.  

We start the investigation with $p=2$. 
% and show that the operator
%$\calQ_t$ is self adjoint on $\bbW^{1,2}_D(\Omega)$. 
Observe that the standard inner product on $\bbW^{1,2}_D(\Omega)$
satisfies $(u,v)_{1,2}= (\calL u, \calL
v)_{0,2} 
=\langle \calP(u),v\rangle $. Hence, by (A3)  
the following identities are valid for $u,v\in
\bbW^{1,2}_D(\Omega)$:
\begin{multline*}
(\calQ_t u, v)_{1,2}= \langle ( \calP - t \calA)u,v\rangle 
= \overline{\langle ( \calP - t \calA)v,u\rangle}
=\overline{(\calP^{-1} ( \calP - t \calA) v,u)_{1,2}} = (u, \calQ_t
  v)_{1,2}. 
%=\big((\ID - t\bbA)\calL u, \calL v\big)_{0,2} 
%\\
%=\big(\calL u,  (\ID - t\bbA)\calL v)_{0,2} = (u,\calQ_t v)_{1,2}.
\end{multline*}
This shows that  $\calQ_t$ is self adjoint on $\bbW^{1,2}_D(\Omega)$. 
Moreover, taking into account the 
 upper bound $M$ of the coefficient matrix $\bbA$ and the uniform
 ellipticity property, the following
 estimates are valid for all $u\in \bbW^{1,2}_D(\Omega)$:
\begin{align*}
(\calQ_t u,u)_{1,2} &= \langle (\calP - t\calA) u,u\rangle \geq
(1-tM)\norm{u}^2_{\bbW^{1,2}(\Omega)}\\
(\calQ_t u,u)_{1,2} &\leq 
(1-t\kappa)\norm{u}^2_{\bbW^{1,2}(\Omega)}.
\end{align*}
Thus, the operator norm  $\norm{\calQ_t}_{\text{op},2}$ with respect to
$\bbW^{1,2}_D(\Omega)$ can be estimated as
\begin{align*}
\norm{\calQ_t}_{\text{op},2} 
&= \sup\{ \abs{\langle \calQ_t u,u\rangle}\,;\, u\in
\bbW^{1,2}_D(\Omega),\, \norm{u}_{\bbW^{1,2}(\Omega)}\leq 1 \}\\
&\leq \max\{\abs{1- tM}, \abs{1-t\kappa}\}.  
\end{align*}
%
%Since the operator norm of a self adjoint operator is given by
%$\norm{\calT}_{\text{op},2} = \sup_{\lambda\in \sigma(T)}\abs{\lambda}$,
%cf.~\cite{ }. 
%and since from the above estimates it follows that
%the spectrum $\sigma(\calQ_t)$ is contained in  $(1-tM,1-t\kappa)$,
Hence,  the operator $\calQ_t$ is a strict 
contraction provided that $t\in {]0, 2 /M[}$. We choose now $t_0
=2/(\kappa +M)$ and define $\wt 
\iota = 1 - t_0\kappa=(M-\kappa)/(M+\kappa)$. With this, we have 
$\norm{\calQ_{t_0}}_{\text{op},2}\leq \wt \iota<1$. 

For $p\in [2,q_1^*]$, interpolation theory  gives  the estimate 
%\begin{align*}
$\norm{\calQ_{t_0}}_{\text{op},p}\leq \wt\iota^{1-\theta} 
\norm{\calQ_{t_0}}_{\text{op},q_1^*}^\theta
$, 
%\end{align*}
where $1/p =(1-\theta)/2 + \theta/q_1^*$. Hence, there exist
$\iota_1\in {]0,1[}$ and $q_1\in {]2,q_1^*]}$ such that for all $p\in [2,q_1]$
it holds  $\norm{\calQ_{t_0}}_{\text{op},p}\leq \iota_1$. Similar arguments
applied to the interval $[q_0^*,2]$ finally imply \eqref{dk_est_p1}. 

Now, we  proceed analogously to the arguments in the proof of Theorem 1 in
\cite{groeger89}: %Groeger
Since the operator $\calQ_{t_0}$ is a contraction on
$\bbW^{1,p}_D(\Omega)$, for every $f\in \bbW_D^{-1,p}(\Omega)$ the
operator $v\mapsto \calQ_{t_0}(v)  + t_0\calP^{-1}f$ has a unique fixed point
$u_f$. Observe that $u_f$ satisfies $\calA u_f=f$. Hence, for all
$p\in[p_0,p_1]$ the operator $\calA$ is a topological isomorphism with
respect to $\bbW_D^{1,p}(\Omega)$. Finally, since 
\begin{align*}
\norm{u_f}_{\bbW^{1,p}(\Omega)} 
=\norm{\calQ_{t_0} u_f + t_0\calP^{-1}f}_{\bbW^{1,p}(\Omega)} 
\leq \iota \norm{u_f}_{\bbW^{1,p}(\Omega)} + t_0
c_{q_0^*,q_1^*}\norm{f}_{\bbW_D^{-1,p}(\Omega)},
\end{align*}
the operator norm of $\calA^{-1}$ is uniformly bounded on $[q_0,q_1]$,
which is \eqref{dk_est_thm1}. 
\end{proof}

%\DDDS EINSCHUB: AMMANN, alles geht auch mit reellwertigen vektoren
%$u:\Omega\rightarrow \R^m$ 
%\DDDE

\begin{exmp}
The equations of linear elasticity as well as the Cosserat-model fit
into this framework. 
In the case of linear elasticity, the vector-function
$u:\Omega\rightarrow \R^d$ (i.e.\@ $m=d$) denotes the displacement
field. Typically, the Dirichlet-boundary is the same for all
components of $u$. Hence, we define 
 $\bbW^{1,p}_D(\Omega)= \prod_{i=1}^d
W_D^{1,p}(\Omega)$, where $D\subset\partial\Omega$ is a closed
$(d-1)$-set. 
The operator of linear elasticity is defined through the
 form $\langle \calA u,v \rangle =\int_\Omega \bfC e(u) : e(v) \; \mathrm{d}
 \mathrm{x}$
 for $u,v\in \bbW^{1,2}_D(\Omega)$, Here, $e(u)=\frac{1}{2}(\nabla
 u + \nabla u^\top)$ is the symmetrised gradient and $\bfC\in
 L^\infty(\Omega; \text{Lin} (\R^{d\times d}_\text{sym},\R^{d\times
   d}_\text{sym}))$ denotes the fourth order elasticity tensor. 
It is assumed that $\bfC$ is symmetric and positive definite on the
symmetric matrices:  for all $F_1,F_2\in \R^{d\times
  d}_\text{sym}$ it holds 
\begin{align*}
\bfC F_1:F_2 =\bfC F_2:F_1,\quad \bfC F_1 :F_1\geq c_\kappa
\abs{F_1}^2.
\end{align*} 
In order to have Korn's second inequality at our disposal, in addition
to (A1) we assume that $\Omega$ is a Lipschitz domain.  Korn's second inequality states that 
the standard norm in $\bbW^{1,2}_D(\Omega)$ and the norm
$\abs{\norm{u}}:=\norm{u}_{L^2(\Omega)} + \norm{e(u)}_{L^2(\Omega)}$
are equivalent, cf.\@ \cite{ell:GS86} and the references therein. 
 Moreover, if $\mathcal H_{d-1}(D)>0$, then standard arguments
relying on the compact embedding of $\bbW^{1,2}_D(\Omega)$ in
$L^2(\Omega)$ show that also  Korn's first inequality is
valid and assumption (A2) is satisfied. %, \cite{ }. %???
 Hence, %if $\meas(D)>0$, then assumption (A2) is satisfied and
Theorems \ref{dk_thm1}  and \ref{dk_thm2} are applicable. 

In the Cosserat models, additionally to the displacement fields the
skew symmetric Cosserat-microrotation-tensor $R\in \R^{3\times
  3}_\text{skew}$  plays a role. 
 Via the relation 
\begin{align*}
\axl R:= \axl \left(\begin{smallmatrix} 0 & r_1 &r_2\\
    -r_1 & 0& r_3\\ -r2&-r3&0\end{smallmatrix}\right):=
\left(\begin{smallmatrix} -r_3 \\r_2 \\-r_1\end{smallmatrix}\right), 
\end{align*}
$\R^{3\times 3}_{\text{skew}}$ is identified with $\R^3$.  
Assume that $D_\text{el},\,
D_\text{R}\subset \partial\Omega$ are nonempty, closed $2$-sets
describing the Dirichlet boundary for the displacements and the tensor
$R$, respectively.  
The state space is defined as $\bbW^{1,p}_D(\Omega)= \prod_{i=1}^3
W_{D_\text{el}}^{1,p}(\Omega) \times \prod_{i=1}^3
W_{D_\text{R}}^{1,p}(\Omega)$.   
A typical differential operator occurring in the theory of Cosserat
models is given by the following weak form for $(u,R),(v,Q)\in
\bbW^{1,2}_D(\Omega)$: 
\begin{multline*}
\langle \calA \left(\begin{smallmatrix} u \\R\end{smallmatrix}
\right),  \left(\begin{smallmatrix} v \\ Q\end{smallmatrix}
\right)\rangle 
 = \int_\Omega 2\mu e(u ): e( v ) + \lambda \dive u \dive v
\\
+ 2\mu_c \dskew (\nabla u -R):\dskew (\nabla v -Q) 
+ \gamma  \nabla \axl R :   \nabla \axl Q \dx. 
\end{multline*}
If in addition to (A1) the domain is a Lipschitz domain and 
if for the  Lam{\'e}-constants $\lambda,\mu$, the Cosserat-couple modulus
$\mu_c$ and the parameter $\gamma$ it holds $\mu>0$, $2\mu + 3\lambda
>0$, $\mu_c\geq 0$ and $\gamma>0$, then condition (A2) is
satisfied, see  \cite{neff06,kne:NK08}, %Neff
 where also more general situations are discussed. Obviously, (A3) is
 satisfied as well, and hence Theorems 
\ref{dk_thm1}  and \ref{dk_thm2} are applicable. 
\end{exmp}

\begin{rem}
We finally remark that on the basis of the previous example the
results from \cite{HMW11} for nonlinear elasticity models 
 can be extended to the situation discussed here by repeating the arguments in
\cite[Section 3]{HMW11}. 
%,
%based on \cite{groeger89}.  
\end{rem}

\section{Applications} \label{sec-appl}
In this chapter we intend to indicate possible applications, which were the
original motivation for this work.

It is more or less clear that the results of this paper cry for applications
 primarily in spatially two-dimensional elliptic/parabolic 
problems. We suggest that in almost all applications resting on
\cite{groeger89} the geometric conditions can be relaxed to those of this
paper, and the results still hold, (see e.g. \cite{liu/cheng/naka},
\cite{de/ge/ju}, \cite{juen}, \cite{miel}, 
\cite{con/mun}, \cite{fa/ga}, \cite{gli/hue}, \cite {ga/gr}, \cite{k/n/r},
 \cite{thermist} to name only a few).

Moreover, the generator property for an analytic semigroup gives the
opportunity to deal also with parabolic problems. When employing the main
result from \cite{TomJo} and then applying the classical semilinear theory,
see e.g.\@ \cite[Ch.~3]{henry}, one should be able to treat also semilinear
ones. Generally, the $W^{-1,q }_D$-calculus allows for right hand sides of the
equations which contain distributional objects as e.g.\@ surface densities
which still belong to the space $W^{-1,q}_D(\Omega)$. In particular, in the
$2d$-case one may even admit functions in time which take their values in the
space of Borel measures, since the space of these measures then continuously
embeds into any space $W^{-1,q}_D(\Omega)$ with $q<2$, compare also
\cite{Amannmeas}.

Moreover, the elliptic regularity result enables a simpler treatment of
problems which include quadratic gradient terms: the a priori knowledge
$\nabla u \in L^q$ with $q>2$ improves the standard information $|\nabla u|^2
\in L^1$ to $|\nabla u|^2 \in L^r$ with $r>1$ -- what makes the analysis of
such problems easier, compare \cite{thermist,KnRoZan11}.

At the end, let us sketch an idea how one can exploit the gain in elliptic
regularity in a rather unexpected direction: Let $q > 2$ be a number such that
\eqref{e-topiso} is a topological isomorphism and \eqref{e-topiso} is also a
topological isomorphism if $\mu$ is there replaced by the adjoint coefficient
function, then providing the adjoint operator in $L^2(\Omega)$. We abbreviate
$A := \nabla \cdot \mu \nabla|_{L^2(\Omega)} $. As in \cite{TomJo}
one can prove that the semigroup operators $e^{tA}$ possess kernels which admit upper
Gaussian estimates. Obviously, these kernels are bounded, and, consequently,
all semigroup operators are Hilbert-Schmidt and even nuclear. Consequently,
$e^{\frac {t}{3}A}:L^2(\Omega) \to L^2(\Omega)$ admits a representation
\[ e^{\frac{t}{3} A} \psi = \sum_j \lambda_j \langle \psi, f_j
	\rangle_{L^2(\Omega)} \; g_j
\]
with $\|f_j\|_{L^2(\Omega)} = \|g_j\|_{L^2(\Omega)} = 1$ and $\sum_j
|\lambda_j| < \infty$, see \cite[Thm.~1.b.3]{koenig}. Hence, $e^{{t}A}$ admits
the following representation via an integral kernel.
\begin{equation} \label{e-integrakern}
  e^{tA} = \sum_j \lambda_j \langle e^{\frac{t}{3} A} \cdot, f_j
	\rangle_{L^2(\Omega)} \; e^{\frac{t}{3} A} g_j = \sum_j \lambda_j
	\; e^{\frac{t}{3} A} g_j \otimes \overline{e^{\frac{t}{3} A^\star}f_j}.
\end{equation}
Let us estimate the $W^{1,q}$-norm of the elements $e^{\frac{t}{3} A} g_j$ and ${e^{\frac{t}{3} A^\star}f_j}$, respectively:
\[ \bigl\| e^{\frac{t}{3} A} g_j \bigr\|_{W^{1,q}_D(\Omega)} \le \|
	(-A + 1)^{-1} \|_{\mathcal L(L^q(\Omega);W^{1,q}_D(\Omega))} \bigl\|
	e^{\frac{t}{6} A} \bigr\|_{\mathcal L(L^2(\Omega);L^q(\Omega))} \bigl\|
	(-A + 1) e^{\frac{t}{6} A} \bigr\|_{\mathcal L (L^2(\Omega))},
\]
since $\|g_j\|_{L^2(\Omega)}=1$. Let us discuss the factors on the right hand
side: the first is finite due to our supposition on $q$ and the embedding
$L^q(\Omega) \hookrightarrow W^{-1,q}_D(\Omega)$. The second is finite because
the semigroup operators are integral operators with bounded kernels. The third
factor is bounded because $A$ generates an analytic semigroup on
$L^2(\Omega)$.

The estimate for $e^{\frac{t}{3} A^\star} f_j$ is quite analogous, this time
investing the continuity of $(-A^*+1)^{-1} : L^q(\Omega) \to
W^{1,q}_D(\Omega)$. Bringing now into play the summability of the series
$\sum_j |\lambda_j|$, one obtains the convergence of the series $\sum_j
\lambda_j \; e^{\frac{t}{3} A} g_j \otimes \overline{e^{\frac{t}{3} A^\star}
f_j}$ in the space $W^{1,q}(\Omega \times \Omega)$. Thus, the semigroup
operators have kernels which are even from $W^{1,q}(\Omega \times \Omega)$
with $q > 2$. We will discuss the consequences of this in a forthcoming paper.
%%%%%%%%%%%%%%%%%%%%%%%%%%%%%%%%%%%%%%%%%%%%%%%%%%%%%%%%%%%%%%%%%%%%%%
%%%%%%%%%%%%%%%%%%%%%%%%%%%%%%%%%%%%%%%%%%%%%%%%%%%%%%%%%%%%%%%%%%%%%%

\section{Concluding remarks}
\begin{rem} \label{r-anwend}
 \begin{enumerate}
  \item As the examples in Figure~\ref{fig-Meissel-1} and \ref{fig-Meissel}
	suggest, admissible constellations for the domain $\Omega$ are finite
	unions of (suitable) Lipschitz domains, whose closures have nonempty
	intersections. Thus, generically, the boundary is the finite union of
	$(d-1)$-dimensional Lipschitz manifolds with the corresponding
	boundary measures.
 \item The $W^{1,p}$-regularity result is also of use for the analysis of
	four-dimensional elliptic equations with right hand side from
	$W^{-1,p}_D(\Omega)$, $p >4$. Namely, the information that the
	solution a priori belongs to a space $W^{1,q}_D$ with $q > 2$, allows
	to localise the elliptic problem within the same class of right hand
	sides, cf.\@ \cite{h/m/r/s}.
%  \item The regularity results on the spaces $W^{1,p}_D(\Omega)$ in case of
% 	$p<2$ provide a frame where spatially two dimensional elliptic and
% 	parabolic equations with measure-valued right hand sides can be
% 	treated. This rests on the fact that in case of two space dimensions
% 	the space of bounded Radon measures on $\overline \Omega$ continuously
% 	embeds into any space $W^{-1,p}_D(\Omega)$ if only $p < 2$, compare
% 	also \cite{Amannmeas}.
 \end{enumerate}
\end{rem}
{\bf Acknowledgment}
In an ealier version of this paper we used Assumption \ref{a-extend} in order
to establish our geometric setting in view of the extension operator, cf.
\cite{TomJo}.
%the same geometric conditions as in \cite{TomJo} to establish the required 
%extension operator $\mathfrak E$. 
However, we learnt from a referee report on this paper that it should be possible to
extend this to more general settings.\\
Secondly, we wish to thank Moritz Egert (Darmstadt) for pointing out to us
Lemma \ref{l-moritz}.
%%%%%%%%%%%%%%%%%%%%%%%%%%%%%%%%%%%%%%%%%%%%%%%%%%%%%%%%%%%%%%%%%%%%%%
%%%%%%%%%%%%%%%%%%%%%%%%%%%%%%%%%%%%%%%%%%%%%%%%%%%%%%%%%%%%%%%%%%%%%%

%\newpage

\end{document}